\pdfoutput=1
\documentclass[11pt,twoside]{article}
\usepackage[T2A,T1]{fontenc} 
\usepackage[utf8]{inputenc}
\usepackage[russian,english]{babel}
\usepackage[a4paper,inner=2.45cm,outer=2.45cm,top=2.35cm,bottom=2.65cm,
  headheight=15pt,headsep=0.55cm,footskip=1.05cm]{geometry}
\usepackage{lmodern}
\usepackage{amsmath,amssymb,mathtools,amsthm}
\usepackage[numbers,sort&compress]{natbib}

\usepackage{booktabs,tabularx,array}
\usepackage{graphicx}
\usepackage{float}
\usepackage{placeins}
\usepackage[ruled,linesnumbered,noline,procnumbered,nokwfunc]{algorithm2e}
\usepackage{enumitem}
\usepackage{xcolor}
\usepackage{microtype}
\usepackage{fancyhdr}
\usepackage{titlesec}

\usepackage[colorlinks=true,linkcolor=blue!55!black,citecolor=blue!55!black,
  urlcolor=blue!55!black]{hyperref}

\ifdefined\pdfsuppresswarningpagegroup
\fi

\definecolor{journalred}{RGB}{176,35,38}
\DontPrintSemicolon
\SetKwInput{KwIn}{Require}
\SetKwInput{KwOut}{Ensure}
\SetAlgoProcName{Procedure}{Procedure}
\allowdisplaybreaks
\numberwithin{equation}{section}
\titleformat{\section}{\large\bfseries}{\thesection.}{0.65em}{}
\titleformat{\subsection}{\normalsize\bfseries}{\thesubsection.}{0.65em}{}
\titleformat{\subsubsection}{\normalsize\bfseries}{\thesubsubsection.}{0.65em}{}
\titlespacing*{\section}{0pt}{2.0ex plus .5ex minus .2ex}{1.0ex}
\titlespacing*{\subsection}{0pt}{1.6ex plus .4ex minus .2ex}{0.7ex}

\newtheorem{theorem}{Theorem}[section]
\newtheorem{lemma}[theorem]{Lemma}
\newtheorem{proposition}[theorem]{Proposition}
\newtheorem{corollary}[theorem]{Corollary}
\theoremstyle{definition}
\newtheorem{assumption}[theorem]{Assumption}

\theoremstyle{remark}

\newcommand{\R}{\mathbb{R}}

\newcommand{\X}{\mathcal{X}}
\newcommand{\ip}[2]{\langle #1,#2\rangle}
\newcommand{\norm}[1]{\lVert #1\rVert}
\newcommand{\dnorm}[1]{\lVert #1\rVert_*}
\newcommand{\Anorm}[1]{\lVert #1\rVert_A}
\newcommand{\Adnorm}[1]{\lVert #1\rVert_{A^{-1}}}
\newcommand{\KL}{\operatorname{KL}}
\newcommand{\ri}{\operatorname{ri}}
\newcommand{\dom}{\operatorname{dom}}
\newcommand{\argmin}{\operatorname*{arg\,min}}
\newcommand{\prox}{\operatorname{prox}}
\DeclareRobustCommand{\TIGER}{\textnormal{\textsc{Tiger}}}
\DeclareRobustCommand{\TIGERbold}{\textbf{\textsf{TIGER}}}
\newcolumntype{Y}{>{\raggedright\arraybackslash}X}

\newcommand{\shorttitle}{\TIGER: Tunable-Geometry Gradient Sliding}
\newcommand{\shortauthor}{D.\,M.\,Dvinskikh,  A.\,G.\,Novitskii,  E.\,R.\,Obzherin, A.\,V.\,Gasnikov}

\fancypagestyle{journal}{%
  \fancyhf{}
  \fancyhead[LE]{\small\bfseries\thepage}
  \fancyhead[RO]{\small\bfseries\thepage}
  \fancyhead[CE]{\small\itshape\shortauthor}
  \fancyhead[CO]{\small\itshape\shorttitle}
  
  }
\fancypagestyle{firstpage}{%
  \fancyhf{}
  
  }
\fancypagestyle{plain}{%
  \fancyhf{}
  \fancyhead[LE]{\small\bfseries\thepage}
  \fancyhead[RO]{\small\bfseries\thepage}
  \fancyhead[CE]{\small\itshape\shortauthor}
  \fancyhead[CO]{\small\itshape\shorttitle}
  
  }
\renewenvironment{abstract}{%
  \begin{center}\begin{minipage}{0.91\textwidth}\small
  \setlength{\parindent}{1.55em}}{%
  \end{minipage}\end{center}}

\hypersetup{
  pdftitle={TIGER: Tunable-Geometry Accelerated Gradient Sliding with Adaptive Safeguards under Similarity},
  pdfauthor={Darina M. Dvinskikh, Anton G. Novitskii, Ewsey R. Obzherin and Alexander V. Gasnikov},
  pdfsubject={Distributed composite optimization and Bregman gradient sliding},
  pdfkeywords={distributed optimization, gradient sliding, adaptive backtracking, tunable geometry, Bregman geometry, similarity}}

\begin{document}
\hypersetup{pageanchor=false}
\thispagestyle{firstpage}
% \makejournalnotice

% \noindent\small MSC 2020: 90C25, 90C06, 68W15\par
% \vspace{2.3em}

\begin{center}
{\LARGE\bfseries \TIGERbold{}: Tunable-Geometry Accelerated Gradient Sliding
with Adaptive Safeguards under Similarity\par}
\vspace{1.5em}
{\large \bfseries
D.\,M.\,Dvinskikh,  A.\,G.\,Novitskii,  E.\,R.\,Obzherin, A.\,V.\,Gasnikov\par}
\end{center}

\begin{abstract}
We study distributed composite optimization under function similarity when
synchronized gradient queries dominate local computation. We propose
\TIGER{}, an accelerated gradient-sliding method in problem-adapted Bregman
geometry. The method applies to constrained non-Euclidean domains, allows
the communicated correction term to be nonconvex, and terminates each local
subproblem using a computable stationarity residual. Each iteration reuses
one synchronized correction gradient during the local solve and makes a
second query to construct the composite subgradient for the mirror step.
Under relative strong convexity and cocoercivity of the stabilized model
error, the communication complexity is
$O((1+\sqrt{\delta_{\omega}/\mu})\log(1/\varepsilon))$, where $\mu$ is the
relative strong-convexity parameter and $\delta_{\omega}$ is a
geometry-adjusted similarity parameter. The analysis recovers the
Hessian-similarity rate in Hilbert geometry and provides an explicit finite
bound on $\delta_{\omega}$ in entropy geometry on a truncated simplex. We
also introduce Adaptive \TIGER{}, which selects the stabilization coefficient
and the similarity estimate during the run. Under uniform backtracking
bounds, the adaptive method has the same asymptotic communication order
as fixed-parameter \TIGER{}, including the queries spent on backtracking.
Numerical experiments compare Adaptive \TIGER{} with several applicable
first-order methods on entropic Wasserstein barycenter problems
using the Fashion-MNIST dataset. 
\end{abstract}

\noindent\textbf{Keywords:} distributed optimization; gradient sliding;
 Bregman geometry; adaptive backtracking; similarity; Wasserstein barycenters.

\par\medskip
\noindent\textit{Darina M. Dvinskikh}\\
 HSE University,\\
% Pokrovsky Blvd. 11, Moscow, 109028 Russia;\\
e-mail: \texttt{dmdvinskikh@hse.ru}
\\[2mm]
\textit{Anton G. Novitskii}\\
Moscow Institute of Physics and Technology,\\
Trusted AI Research Center RAS, \\
% 109004 Moscow, Russia
% Institutskiy per. 9, Dolgoprudny, 141701 Russia\\
e-mail: \texttt{novitskii.ag@phystech.edu}
\\[2mm]
\textit{Ewsey R. Obzherin}\\
Moscow Institute of Physics and Technology, \\
FusionBrain Lab\\
% Institutskiy per. 9, Dolgoprudny, 141701 Russia\\
e-mail: \texttt{obzherin.ewsey@ya.ru}
\\[2mm]
\textit{Alexander V. Gasnikov}\\
Innopolis University, \\ 
Moscow Institute of Physics and Technology,\\
%Institutskiy per. 9, Dolgoprudny, 141701 Russia\\
Institute for Information Transmission Problems RAS,\\
e-mail: \texttt{gasnikov@yandex.ru}
\par\medskip
\noindent\small\textit{Funding.} The research was supported by Russian Science Foundation (project No. 23-11-00229-{\fontencoding{T2A}\fontfamily{cmr}\selectfont П}),
\url{https://rscf.ru/en/project/23-11-00229/}.\normalsize
\par\bigskip
\hrule\bigskip
\setcounter{page}{1}
\hypersetup{pageanchor=true}
\pagestyle{journal}

\section{Introduction}
\label{sec:introduction}

% Distributed first-order methods alternate inexpensive local operations with
% synchronized broadcasts and reductions.  Every synchronized gradient query
% involves all data holders and may wait for network latency or a slow
% participant.  Communication is therefore often the dominant cost even when
% local gradients are cheap.  Curvature similarity makes this imbalance useful:
% local objectives can serve as accurate curvature models, while communication
% is reserved for a smoother global-to-local correction.  This mechanism
% underlies DANE, statistical preconditioning, and optimal gradient sliding
% under similarity
% \citep{shamir2014communication,hendrikx2020statistically,kovalev2022optimal}.

We consider distributed composite optimization over a star network of $n$
agents.  Agent~$i$ holds a differentiable local loss $f_i$, and all agents
seek a common solution of
\begin{equation}
 \min_{x\in\mathcal{X}} \Phi(x):=F(x)+r(x),
 \qquad
 F(x):=\sum_{i=1}^n w_i f_i(x),
 \label{eq:intro-distributed-objective}
\end{equation}
where $w_i>0$, $\sum_{i=1}^n w_i=1$, $\X$ is a nonempty closed convex
set, and $r:\X\to\R$ is an optional finite-valued closed convex
regularizer.  One data-holding agent is designated as the server and
relabeled as agent~$1$.  Evaluations of $f_1$ are local to the server,
whereas an aggregate gradient requires a synchronized broadcast and
reduction.

Our aim is to reduce these synchronized queries when communication is more
expensive than local computation.  This regime arises, for example, in
distributed empirical risk minimization: local losses may have similar
curvature when the agents hold representative samples from the same
data-generating distribution
\citep{shamir2014communication,hendrikx2020statistically}.  For twice
continuously differentiable losses, the familiar Euclidean form of this
similarity is
\begin{equation}
 \sup_{x\in\X}
 \norm{\nabla^2F(x)-\nabla^2f_1(x)}_2\le\delta,
 \label{eq:intro-hessian-similarity}
\end{equation}
where $\norm{\cdot}_2$ is the spectral norm for matrices and $\delta$
measures the curvature discrepancy between the aggregate and server losses.
The non-Euclidean formulation in Section~\ref{sec:setting} replaces this
matrix bound by a condition adapted to the chosen geometry.

We introduce \TIGER{}, a tunable-geometry accelerated gradient-sliding
method.  The name abbreviates \textbf{T}unable geometry, \textbf{I}nexact
local solves, \textbf{G}radient sliding, \textbf{E}xtrapolation, and
\textbf{R}esidual certification.  At each outer iteration, the method keeps a
synchronized correction gradient fixed while solving a regularized local
model at the server.  A second synchronized query reconstructs the composite
subgradient used in the accelerated mirror step.  The distance-generating
function may be selected to reflect the constraint set, and the local solve
is terminated by a computable stationarity residual.  Adaptive \TIGER{}
selects the stabilization and similarity estimates during the run rather
than relying on sharply tuned global values.

\paragraph{Contributions.}
\begin{itemize}
\item \textbf{A new tunable-geometry gradient-sliding algorithm.}
The method extends the separation between synchronized correction queries and
server-local computation to constrained composite problems in a
problem-selected Bregman geometry.  The mirror map need not be induced by the
server loss, and the local subproblem may be solved inexactly under a
computable residual criterion.

\item \textbf{An accelerated communication guarantee.}
Under relative strong convexity and a geometry-adjusted condition on the
model error, \TIGER{} attains synchronized-query complexity
$O((1+\sqrt{\delta_\omega/\mu})\log(1/\varepsilon))$.  Here $\mu$ is the
relative strong-convexity parameter and $\delta_\omega$ measures similarity
in the selected geometry.  The result recovers the standard
Hessian-similarity rate in quadratic Hilbert geometry and yields explicit
finite parameters for entropy geometry on a truncated simplex.

\item \textbf{Adaptive parameter selection.}
Adaptive \TIGER{} updates the stabilization coefficient and the
geometry-adjusted similarity estimate through computable backtracking
conditions.  Under uniform bounds for these searches, it preserves the
asymptotic communication order of the fixed-parameter method, including the
queries used by rejected candidates.

\item \textbf{A non-Euclidean application.}
We specialize the analysis to entropic Wasserstein barycenters and compare
Adaptive \TIGER{} with applicable first-order methods on problems constructed
from Fashion-MNIST images, using synchronized queries as the communication
metric.
\end{itemize}

\subsection{Related work}
\label{sec:related-work}

\paragraph{Similarity-based distributed optimization.}
Similarity between local objectives has been studied extensively as a means
of reducing communication in distributed optimization.  An early
server-based example is DANE, introduced by
\citet{shamir2014communication}.  DANE corrects each local objective using
the discrepancy between the local and global gradients and then averages the
resulting local minimizers.  Related Newton-type approaches include DiSCO,
proposed by \citet{zhang2015disco}, which computes inexact damped Newton steps
using distributed preconditioned conjugate gradients, and AIDE, introduced
by \citet{reddi2016aide}, which applies Catalyst acceleration to an inexact
DANE scheme.

Subsequent work has considered similarity in server-based, federated, and
gradient-tracking models
\citep{tian2022acceleration,sun2022distributed,khaled2023faster,
takezawa2025exploiting}, as well as further separations between local
computation and communication
\citep{bylinkin2024aseg,bylinkin2026separation}.  The attainable
communication rate depends on the permitted local computation and the
underlying oracle model, as reflected in the corresponding lower bounds
\citep{arjevani2015communication,scaman2017optimal,kovalev2021lower,
kovalev2022optimal}.  Decentralized optimization over time-varying networks
uses a different communication model and requires additional consensus
mechanisms
\citep{rogozin2023decentralized,metelev2023consensus}.

\paragraph{Accelerated methods under similarity.}
\citet{hendrikx2020statistically} propose SPAG, an accelerated distributed
method that uses a regularized empirical loss stored at the server to define
a loss-induced Bregman geometry.  Its convergence rate is governed by the
relative smoothness and relative strong convexity of the global objective
with respect to this reference function, which is assumed to be smooth and
strongly convex in a Euclidean norm.  For a general nonquadratic reference
function, the finite-time bound also depends on a triangle-scaling factor.
A pure square-root dependence on the relative condition number is obtained
for a quadratic reference function and asymptotically under a
Lipschitz-Hessian assumption.  SPAG estimates the triangle-scaling gain
through an acceptance condition evaluated at each iteration, while keeping the relative
smoothness and strong-convexity parameters fixed.

\citet{kovalev2022optimal} develop Optimal Gradient Sliding (OGS), an
oracle-separation method for distributed optimization under similarity.
OGS keeps the gradient of a smooth communicated component fixed during an
inexact auxiliary solve and allows this component to be nonconvex.  The
method attains the optimal square-root dependence on the Euclidean
similarity constant.  Its analysis, however, relies on squared-Euclidean
regularization and does not directly cover the constrained composite
Bregman setting considered here.  The Bregman proximal method of
\citet{beznosikov2024bregman} allows a general problem geometry under
similarity for stochastic monotone variational inequalities, but does not
establish an accelerated square-root communication bound for the strongly
convex composite minimization problem studied in this paper.

\TIGER{} combines oracle separation with a problem-selected Bregman geometry
and a residual-certified composite local solve.  Its adaptive version
differs from the gain adaptation in SPAG.  Adaptive \TIGER{} updates the
stabilization coefficient $\beta$ and the geometry-adjusted similarity
estimate that determines the acceleration weight, using separate acceptance
conditions for the local model and the accelerated mirror step.

\paragraph{Gradient sliding and Bregman acceleration.}
Gradient sliding was introduced to reduce evaluations of an expensive
component of a composite objective by reusing its gradient during inner
iterations \citep{lan2016gradient}.  Accelerated variants extend this form of
oracle separation to broader composite settings
\citep{lan2022accelerated}.

From a geometric perspective, \citet{chen2026accelerated} study accelerated
mirror descent under a four-point compatibility condition between the
smooth objective and the mirror geometry.  Their method does not distinguish
synchronized correction queries from computations performed locally at the
server.  Our analysis uses a different requirement: the geometric
compatibility condition is imposed only on a stabilized discrepancy between
the global objective and a reference loss stored at the server, rather than
on the full objective.  We show that this condition is sufficient for the
four-point estimate used in the convergence analysis.  This formulation
preserves the separation between synchronized queries and local computation
and allows the local subproblems to be terminated using a computable
residual.

Accelerated proximal schemes and their inexact variants provide natural
composite benchmarks when the required proximal map is available
\citep{tseng2008accelerated,beck2009fista,nesterov2013composite,
schmidt2011inexact,taylor2017composite,upadhyaya2026optimal}.

\paragraph{Wasserstein barycenters.}
Wasserstein barycenters are Fréchet means of probability measures in
Wasserstein space \citep{agueh2011barycenters}.  Algorithms for fixed-support
barycenters with entropic regularization were developed by
\citet{cuturi2014fast} and \citet{benamou2015iterative}.  The associated
regularized transport problems admit the classical Sinkhorn scaling
representation \citep{cuturi2013sinkhorn}, and the resulting value function
is smooth in the interior of the simplex
\citep[Theorem~2]{luise2018differential}.  Computational aspects of optimal
transport are surveyed by \citet{peyre2019computational}.  Complexity bounds
and distributed formulations of Wasserstein barycenter problems are studied
in
\citep{kroshnin2019complexity,uribe2018distributed,
dvinskikh2021decentralized}.

\paragraph{Organization.}
Section~\ref{sec:setting} introduces the distributed composite problem and
the underlying Bregman geometry.  Section~\ref{sec:algorithm} presents
\TIGER{}  along with its
convergence and synchronized-query complexity guarantees.
Section~\ref{sec:adaptive-tiger} develops adaptive choices of the two
curvature parameters and gives the resulting communication guarantee.
Section~\ref{sec:geometry} derives verifiable model-error bounds in Hilbert
and entropy geometries and applies \TIGER{} to entropic Wasserstein
barycenters.  Section~\ref{sec:experiments} presents the numerical results, and
Section~\ref{sec:conclusion} concludes the paper. The appendices contain the proofs of the main results, the full specification
of Adaptive \TIGER{}, and additional details of the numerical experiments.

\section{Preliminaries}
\label{sec:setting}

\subsection{Similarity}

We solve problem~\eqref{eq:intro-distributed-objective} and assume that its
minimum is attained:
\begin{equation*}
 x^\star\in\argmin_{x\in\X}\Phi(x),
 \qquad
 \Phi^\star:=\Phi(x^\star)=\min_{x\in\X}\Phi(x)\in\R.
\end{equation*}

% At a query point $x$, the server broadcasts $x$ to agents $2,\ldots,n$,
% receives their gradients, evaluates $\nabla f_1(x)$ locally, and forms
% \begin{equation*}
%  \nabla F(x)=\sum_{i=1}^n w_i\nabla f_i(x).
% \end{equation*}
% We count this broadcast and reduction as one synchronization round.  By
% contrast, evaluating $f_1+r$ requires no communication.

For the fixed-metric results, assume additionally that the losses are twice
continuously differentiable on an open neighborhood of $\X$.  Let
$A=A^\top\succ0$ and define
\begin{equation*}
 \Anorm{x}:=\sqrt{x^\top Ax},
 \qquad
 \Adnorm{s}:=\sqrt{s^\top A^{-1}s}
 =\sup_{\Anorm{u}\le1}\ip{s}{u}.
\end{equation*}
Thus, in the fixed quadratic geometry, the generic notation specializes to
$\norm{\cdot}=\Anorm{\cdot}$ and
$\dnorm{\cdot}=\Adnorm{\cdot}$.  In general, the dual norm associated with
$\norm{\cdot}$ is
\[ \dnorm{s}:=\sup_{\norm{u}\le1}\ip{s}{u}.\]
For a linear operator $H:\R^d\to\R^d$, let
\begin{equation*}
 \norm{H}_{A\to A^{-1}}
 :=\sup_{\Anorm{u}\le1}\Adnorm{Hu}.
\end{equation*}
This is the operator norm from the primal space
$(\R^d,\Anorm{\cdot})$ to its dual $(\R^d,\Adnorm{\cdot})$.
In this metric, we use the standard Hessian similarity condition
\citep{hendrikx2020statistically,kovalev2022optimal}
\begin{equation}
 \sup_{x\in\X}
 \norm{\nabla^2F(x)-\nabla^2f_1(x)}_{A\to A^{-1}}
 \le\delta.
 \label{eq:reference-hilbert-similarity}
\end{equation}
Since the Hessian difference is symmetric,
condition~\eqref{eq:reference-hilbert-similarity} is equivalent to
\begin{equation*}
 -\delta A
 \preceq \nabla^2F(x)-\nabla^2f_1(x)
 \preceq \delta A,
 \qquad x\in\X.
\end{equation*}
For $A=I$, the displayed Hessian bound is exactly
\eqref{eq:intro-hessian-similarity}.  The parameter $\delta$ is used only in
the quadratic Hilbert specialization.  The general Bregman result uses the
cocoercivity condition on the model-error gradient introduced below.

% \subsection{Correction split and relative geometry}

The similarity condition suggests separating the globally communicated
correction from the objective available at the server.  Define
\begin{equation*}
 \Phi=p+q,
 \qquad
 p:=F-f_1,
 \qquad
 q:=f_1+r.
\end{equation*}
Thus $p$ is the smooth correction whose gradient requires synchronization,
whereas $q$ is the composite term available at the server.  When taking
subdifferentials, we extend $q$ by $+\infty$ outside $\X$ and assume that this
extension is proper, closed, and convex.  This holds, for example, when
$f_1$ is convex on $\X$.  The correction $p$ is continuously differentiable
on $\X$ and may be nonconvex.  Throughout, $\partial q$ denotes the convex
subdifferential of the extended function.
% One correction query returns
% \begin{equation}
%  \nabla p(x)=\nabla F(x)-\nabla f_1(x)
%  =\sum_{i=2}^n w_i\bigl(\nabla f_i(x)-\nabla f_1(x)\bigr).
%  \label{eq:correction-query}
% \end{equation}
Consequently, condition~\eqref{eq:reference-hilbert-similarity} is
equivalent to the Hessian bound
$\norm{\nabla^2p(x)}_{A\to A^{-1}}\le\delta$.  Integrating the Hessian
along the segment from $y$ to $x$ gives the gradient consequence
\begin{equation*}
 \Adnorm{\nabla p(x)-\nabla p(y)}
 \le\delta\Anorm{x-y},
 \qquad x,y\in\X,
\end{equation*}
which is the form used in the Hilbert-space estimates below.

Although $p$ may be nonconvex, the analysis only uses vectors in the
set $\nabla p(x)+\partial q(x)$.

\subsection{Bregman divergence and distance-generating function} 
Let $\omega:\X\to\R$ be a distance-generating function: it is convex,
continuously differentiable, and $1$-strongly convex with respect to
$\norm{\cdot}$,
\begin{equation}
 \omega(y)\ge \omega(x)+\ip{\nabla\omega(x)}{y-x}
 +\frac12\norm{y-x}^2
 \qquad(x,y\in\X).
 \label{eq:omega-strong}
\end{equation}
Its Bregman divergence is
\begin{equation*}
 D_\omega(y,x):=\omega(y)-\omega(x)-\ip{\nabla\omega(x)}{y-x}.
\end{equation*}
Thus $D_\omega(y,x)\ge\norm{y-x}^2/2$.

This Bregman divergence supplies the problem-adapted stabilization and mirror
displacement used below.  In Euclidean geometry,
$ \omega(x)=\frac12\norm{x}_2^2,$ and
$ D_\omega(y,x)=\frac12\norm{y-x}_2^2.$
More generally, $\omega(x)=\Anorm{x}^2/2$ gives a fixed metric, while
$\omega(x)=\sum_i x_i\log x_i$ gives the entropy geometry on a simplex.

Relative strong convexity compares objective curvature with a
distance-generating function rather than with a squared Euclidean norm
\citep{lu2018relatively}.  The following assumption uses its standard
composite extension; related relative-conditioning assumptions underlie
statistical preconditioning \citep{hendrikx2020statistically}.  For
$\omega(x)=\norm{x}_2^2/2$, it reduces to ordinary $\mu$-strong convexity.

\begin{assumption}[Relative strong convexity]
\label{ass:outer}
There is $\mu>0$ such that, for every $x,y\in\X$ and
$s\in\partial q(x)$,
\begin{equation*}
 \Phi(y)\ge \Phi(x)+\ip{\nabla p(x)+s}{y-x}+\mu D_\omega(y,x).
\end{equation*}
\end{assumption}

\section{The \TIGER{} method}
\label{sec:algorithm}
We present \TIGER{} and prove an accelerated linear rate under the
prescribed local residual condition.

\subsection{Bregman proximal-gradient subproblem}

Fix a stabilization parameter $\beta\ge0$.  At a synchronized query point
$y\in\X$, define
\begin{equation}
 \mathcal M_y(x)
 :=p(y)+\ip{\nabla p(y)}{x-y}+q(x)+\beta D_\omega(x,y),
 \qquad x\in\X.
 \label{eq:local-model}
\end{equation}
This is the Bregman proximal-gradient subproblem obtained by linearizing $p$,
keeping $q$ intact, and adding a Bregman proximal term.  Once $\nabla p(y)$
has been formed, all iterations used to solve~\eqref{eq:local-model} are local
to the server.  \TIGER{} terminates the solve when a computable optimality
residual reaches its prescribed tolerance.  

\subsection{Controlling the model error}

For the stabilization parameter $\beta$ used in~\eqref{eq:local-model}, define
$R_\beta(x):=\beta\omega(x)-p(x)$.  The identity
\begin{equation}
 \mathcal M_y(x)-\Phi(x)
 =R_\beta(x)-R_\beta(y)
  -\ip{\nabla R_\beta(y)}{x-y},
 \qquad x\in\X,
 \label{eq:model-identity-corrected}
\end{equation}
shows that convexity of $R_\beta$ makes the subproblem objective a majorant of
$\Phi$.  The auxiliary function $R_\beta$ represents the stabilized modeling
error; the algorithmic geometry remains generated by $\omega$.

Acceleration creates a cross term between a change of $\nabla R_\beta$ and
an independent mirror displacement.  The following condition controls this
interaction.

\begin{assumption}[Cocoercivity of the model-error gradient]
\label{ass:correction}
The function $R_\beta$ is convex on $\X$, and there is a finite
$\delta_{\omega}\ge0$ such that, for every $x,y\in\X$,
\begin{equation*}
 \dnorm{\nabla R_\beta(y)-\nabla R_\beta(x)}^2
 \le
 \delta_{\omega}
 \ip{\nabla R_\beta(y)-\nabla R_\beta(x)}{y-x}
 .
\end{equation*}
\end{assumption}

The inner product on the right-hand side is the symmetrized Bregman divergence
of $R_\beta$.  For $\delta_{\omega}>0$,
Assumption~\ref{ass:correction} states that $\nabla R_\beta$ is
$1/\delta_{\omega}$-cocoercive in the selected norm and its dual.  We
refer to $\delta_{\omega}$ as the geometry-adjusted similarity parameter.

\begin{lemma}[Four-point model-error bound]
\label{lem:model-error-four-point}
Under Assumption~\ref{ass:correction}, for every $x,y,a,b\in\X$,
\begin{equation*}
 \left|
 \ip{\nabla R_\beta(y)-\nabla R_\beta(x)}{a-b}
 \right|^2
 \le
 2\delta_{\omega}
 \ip{\nabla R_\beta(y)-\nabla R_\beta(x)}{y-x}
 D_\omega(a,b).
\end{equation*}
\end{lemma}

\begin{proof}
Dual Cauchy--Schwarz, Assumption~\ref{ass:correction}, and
$D_\omega(a,b)\ge\norm{a-b}^2/2$ give the result.
\end{proof}

\paragraph{Relation to accelerated Bregman analyses.}
Lemma~\ref{lem:model-error-four-point} has the same four-point structure as
the compatibility condition used in accelerated mirror methods
\citep{chen2026accelerated}.  Here it follows from cocoercivity of the
gradient of $R_\beta(x)=\beta\omega(x)-p(x)$.  Only $\nabla p$ therefore
requires synchronization, while $q$ remains inside the Bregman
proximal-gradient subproblem.  In contrast to the loss-induced geometry of
SPAG \citep{hendrikx2020statistically}, $\omega$ may be selected independently
of $f_1$.  Proposition~\ref{prop:hilbert-cocoercivity} recovers the fixed-metric
OGS dependence \citep{kovalev2022optimal}.

The choice of the acceleration weight follows directly from the one-step
estimate.  Applying Lemma~\ref{lem:model-error-four-point} and Young's
inequality to
the acceleration cross term produces the coefficient
$\tau^2\delta_{\omega}/\mu$ in front of a Bregman descent term.  The cross term is
therefore absorbed whenever $\tau^2\delta_{\omega}\le\mu$.  Since the contraction
factor $(1+\tau)^{-1}$ improves as $\tau$ increases, while the extrapolation
point must remain a convex combination and hence requires $\tau\le1$, we use
the largest admissible value
\begin{equation}
 \tau:=
 \left(\max\left\{1,\sqrt{{\delta_{\omega}}/{\mu}}\right\}\right)^{-1}.
 \label{eq:tau-choice}
\end{equation}

Assumptions~\ref{ass:outer} and~\ref{ass:correction} also ensure that the
two subproblems below are well posed.  Indeed, $\Phi(x)-\mu\omega(x)$ and
$R_\beta(x)$ are convex, so~\eqref{eq:model-identity-corrected} makes
$\mathcal M_y$ $\mu$-strongly convex.  The outer mirror objective is
$(1+\tau)$-strongly convex.  Closedness and coercivity ensure existence,
while strong convexity ensures uniqueness.

\subsection{Algorithm}
% To evaluate $\nabla p(x)$, the server broadcasts $x$, the agents return their
% local gradients, and the server aggregates the responses.  We count this
% broadcast--aggregation operation as one communication round.

\TIGER{} maintains a mirror state $x_k$ and an accepted local state $v_k$.
The first synchronized correction is reused throughout the local subproblem
solve, which is the gradient-sliding step; the second reconstructs a valid
composite subgradient for the outer Bregman update, while all intermediate
work is local to the server.

\begin{algorithm}[H]
\caption{\TIGER{}}
\label{alg:tiger-corrected}
\KwIn{starting points $x_0=v_0\in\X$; relative strong convexity parameter
$\mu>0$; stabilization parameter $\beta\ge0$; geometry-adjusted similarity parameter
$\delta_{\omega}\ge0$; number of iterations $K\in\mathbb N$; residual tolerances
$\xi_k\ge0$.}

$\tau\leftarrow1/\max\{1,\sqrt{\delta_{\omega}/\mu}\}$;

\For{$k=0,\ldots,K-1$}{
 $y_k\leftarrow\tau x_k+(1-\tau)v_k$\;
 Synchronize at $y_k$ and set $\ell_k\leftarrow\nabla p(y_k)$\;
 Locally find $v_{k+1}\in\X$ and
 $e_k\in\partial\mathcal M_{y_k}(v_{k+1})
 =\ell_k+\partial q(v_{k+1})
 +\beta\bigl(\nabla\omega(v_{k+1})-\nabla\omega(y_k)\bigr)$ such that
 $\dnorm{e_k}^2\le\xi_k$\;
 Synchronize at $v_{k+1}$ and set
 $g_k\leftarrow e_k+\nabla p(v_{k+1})-\ell_k
 -\beta\bigl(\nabla\omega(v_{k+1})-\nabla\omega(y_k)\bigr)$\;
 $x_{k+1}\leftarrow\displaystyle
 \argmin_{x\in\X}\left\{
  \frac{\tau}{\mu}\ip{g_k}{x}
  +\tau D_\omega(x,v_{k+1})+D_\omega(x,x_k)
 \right\}$\;
}
\KwOut{$v_K$.}
\end{algorithm}

The correction $\nabla p(y_k)$ is held fixed while the server solves
$\mathcal M_{y_k}$; this is the gradient-sliding phase.  The second
synchronization gives
$g_k\in\nabla p(v_{k+1})+\partial q(v_{k+1})$ by direct cancellation.
Thus every outer iteration uses two synchronization rounds, and all work
between them is local to the server.

\subsection{Convergence guarantee}

\begin{lemma}[Last-iterate Bregman bound]
\label{lem:rates}
Let Assumptions~\ref{ass:outer} and~\ref{ass:correction} hold.
Let $\{(x_k,v_k)\}_{k=0}^K$ be generated by
Algorithm~\ref{alg:tiger-corrected} with $\tau$ chosen according
to~\eqref{eq:tau-choice} and residual bounds $\xi_k\ge0$.
Then, for every $k=0,\ldots,K$,
\begin{equation*}
 \mu(1+\tau)D_\omega(x^\star,x_k)
 +\Phi(v_k)-\Phi^\star 
 \le \frac{
  \Phi(v_0)-\Phi^\star
  +\mu(1+\tau)D_\omega(x^\star,x_0)}{(1+\tau)^{k}}+\frac1{2\mu}\sum_{j=0}^{k-1}
 \frac{\xi_j}{(1+\tau)^{k-1-j}},
\end{equation*}
where the sum is zero for $k=0$.
\end{lemma}

\noindent\textit{Proof idea.}
The algorithm has two states: the objective gap measures the accepted local
point $v_k$, whereas the Bregman term measures the mirror state $x_k$.  An
objective-gap recursion alone therefore does not close.  For the proof only,
define the nonnegative Lyapunov quantity
\begin{equation}
 \mathcal L_k
 :=\Phi(v_k)-\Phi^\star
 +\mu(1+\tau)D_\omega(x^\star,x_k).
 \label{eq:lyapunov-quantity}
\end{equation}
The coefficient of the Bregman term makes both states contract with the same
factor.  The outer variational inequality, the Bregman three-point identity,
and relative strong convexity give
\begin{equation}
 \mathcal L_{k+1}
 \le \frac{\mathcal L_k}{1+\tau}
 +\frac{\dnorm{e_k}^2}{2\mu}
 \le \frac{\mathcal L_k}{1+\tau}+\frac{\xi_k}{2\mu}.
 \label{eq:lyapunov-recursion}
\end{equation}
Lemma~\ref{lem:model-error-four-point} and
$\tau^2\delta_{\omega}\le\mu$ absorb the model-error cross term; strong
convexity of $\omega$ and Fenchel--Young absorb the residual.  Unrolling
\eqref{eq:lyapunov-recursion} and expanding~\eqref{eq:lyapunov-quantity}
gives Lemma \ref{lem:rates}.  Thus the Lyapunov notation is only a
proof device: the result itself is stated in terms of the objective gap and
the Bregman divergence.  Dropping the nonnegative terminal Bregman term gives
an objective-gap bound.  The complete argument is in
Appendix~\ref{app:proof-one-step}--\ref{app:proof-rates}.

For the communication guarantee, let $\overline B_0$ be any upper bound on
the initial error, written in the original quantities as
\begin{equation}
 \overline B_0\ge
 \Phi(v_0)-\Phi^\star
 +\mu(1+\tau)D_\omega(x^\star,x_0).
 \label{eq:initial-scale}
\end{equation}
If the residual tolerances satisfy
\begin{equation}
 \xi_k=\frac{\mu\tau\overline B_0}{4}
 \left(1-\frac\tau4\right)^{k+1},
 \label{eq:residual-schedule}
\end{equation}
then
\begin{equation}
 \mu(1+\tau)D_\omega(x^\star,x_K)
 +\Phi(v_K)-\Phi^\star
 \le \overline B_0(1-\tau/4)^K.
 \label{eq:residual-rate}
\end{equation}

If $x_0=v_0\in\X$ and
$g_0\in\nabla p(x_0)+\partial q(x_0)$ is available, then
$\overline B_0=4\dnorm{g_0}^2/\mu$ is admissible; see
Appendix~\ref{app:proof-initial-bound}.

\subsection{Communication complexity}

The following consequence of Lemma~\ref{lem:rates} is the principal
synchronized-query guarantee.

% The comparison that follows is placed here,
% rather than in the introduction or problem statement, so that every parameter
% in the \TIGER{} row has already been defined and justified by a theorem.

\begin{theorem}[Communication complexity of \TIGER{}]
\label{thm:communication}
Suppose that the conditions of Lemma~\ref{lem:rates} hold, let
$\overline B_0$ satisfy~\eqref{eq:initial-scale}, and fix
$\varepsilon>0$.  Run Algorithm~\ref{alg:tiger-corrected} with residual
tolerances set according to~\eqref{eq:residual-schedule}.
After
\begin{equation*}
 K=O\left(
  \left(1+\sqrt{\frac{\delta_{\omega}}{\mu}}\right)
  \log\left(1+\frac{\overline B_0}{\varepsilon}\right)
 \right)
\end{equation*}
outer iterations (and hence the same order of synchronized communication
rounds), the output satisfies $\Phi(v_K)-\Phi^\star\le\varepsilon$.
\end{theorem}

% \begin{proof}
% Equation~\eqref{eq:residual-rate} implies
% \[
%  \Phi(v_K)-\Phi^\star
%  \le \overline B_0\exp(-\tau K/4).
% \]
% Indeed, $-\log(1-\tau/4)\ge\tau/4$.  Thus
% $K=O\bigl(\tau^{-1}\log(1+\overline B_0/\varepsilon)\bigr)$ makes the
% right-hand side at most \(\varepsilon\).  Finally,
% $\tau^{-1}=\max\{1,\sqrt{\delta_{\omega}/\mu}\}$, and the two synchronized
% evaluations of \(\nabla p\) in each iteration change the count only by a
% constant factor.
% \end{proof}

\paragraph{Finite local realization.}
The communication theorem requires only the computable residual in
Algorithm~\ref{alg:tiger-corrected}.
Appendix~\ref{app:local-proofs} proves that, in a fixed quadratic metric, a
restarted accelerated proximal-gradient routine reaches this residual in
finitely many local evaluations under uniform smoothness, strong convexity,
an available metric proximal map, and a warm-start bound.  No additional
communication is used by the inner routine.

Table~\ref{tab:method-comparison} compares the assumptions and local
primitives behind the communication bounds.  Fixed and Adaptive \TIGER{} are
listed separately because the adaptive method selects curvature parameters
during the run and counts the queries spent on backtracking.  The quantities
$\kappa_{\rm rel}$ and
$\kappa_\varphi$ are SPAG's relative and mirror-map condition numbers, and
$\rho_{\rm D}$ is DANE's quadratic contraction factor.  For AIDE, the
displayed similarity-dependent guarantee is its quadratic result.
The quantities $\Delta_{\rm ad}$, $L_\beta$, and $L_\delta$ in the adaptive
row are defined in Corollary~\ref{cor:adaptive-communication}.

\begin{table}[ht]
\caption{Communication bounds and structural assumptions under strong
convexity.  Initialization-dependent factors inside
$\log(1/\varepsilon)$ are suppressed.}
\label{tab:method-comparison}
\centering
\footnotesize
\setlength{\tabcolsep}{2.6pt}
\renewcommand{\arraystretch}{1.16}
\begin{tabularx}{\textwidth}{
@{}
>{\raggedright\arraybackslash}p{0.12\textwidth}
>{\raggedright\arraybackslash}p{0.315\textwidth}
>{\raggedright\arraybackslash}p{0.23\textwidth}
>{\centering\arraybackslash}X
@{}}
\toprule
\textbf{Method}
& \textbf{Geometry and problem class}
& \textbf{Local computation}
& \textbf{Communication rounds}\\
\midrule

APG \citep{nesterov2013composite}
& Quadratic Hilbert geometry; convex smooth $p$
& Exact proximal map of $q$
& $O\!\left(
  (1+\sqrt{\delta/\mu})\log(1/\varepsilon)
  \right)$\\
\midrule

DANE \citep{shamir2014communication}
& Euclidean corrected models; classical quadratic guarantee
& Local objectives solved at the agents and averaged
& $O\!\left(
  \log(1/\varepsilon)/\log(1/\rho_{\rm D})
  \right)$\\
\midrule

AIDE \citep{reddi2016aide}
& Euclidean unconstrained quadratics; pairwise
  $\delta$-related local objectives
& Catalyst acceleration with parallel inexact DANE solves
& $\widetilde O\!\left(
  (1+\sqrt{\delta/\mu})\log(1/\varepsilon)
  \right)$\\
\midrule

SPAG \citep{hendrikx2020statistically}
& Loss-induced Bregman map; smooth and strongly convex mirror map
& Exact Bregman-preconditioned server step
& $O\!\left(
  \sqrt{\kappa_{\rm rel}\kappa_\varphi}
  \log(1/\varepsilon)
  \right)$;
  quadratic mirror:
  $O\!\left(
  \sqrt{\kappa_{\rm rel}}\log(1/\varepsilon)
  \right)$\\
\midrule

PAUS \citep{beznosikov2024bregman}
& General Bregman geometry; composite strongly monotone
  variational inequality
& Server-local stochastic composite Mirror--Prox procedure
& $O\!\left(
  (\delta/\mu)\log(1/\varepsilon)
  \right)$ in the deterministic specialization\\
\midrule

OGS \citep{kovalev2022optimal}
& Quadratic Hilbert geometry; $p$ may be nonconvex
& Inexact Euclidean auxiliary solves
& $O\!\left(
  (1+\sqrt{\delta/\mu})\log(1/\varepsilon)
  \right)$\\
\midrule

\TIGERbold{} \textbf{(this paper)}
& Problem-selected Bregman geometry; constrained composite model;
  $p$ may be nonconvex
& Bregman proximal-gradient subproblem with a computable residual
& $O\!\left(
  (1+\sqrt{\delta_{\omega}/\mu})\log(1/\varepsilon)
  \right)$\\
\midrule

\textbf{Adaptive} \TIGERbold{} \textbf{(this paper)}
& Same Bregman composite class; $\beta$ and the similarity estimate are
  selected during the run
& Residual-controlled local solve; secant and model-error absorption conditions
& $O\!\left(
  (1+\sqrt{\Delta_{\rm ad}/\mu})\log(1/\varepsilon)
  \right)$
  $+\,O(L_\beta+L_\delta)$\\
\bottomrule
\end{tabularx}
\end{table}

\section{Adaptive parameter selection for \TIGER{}}
\label{sec:adaptive-tiger}

Algorithm~\ref{alg:tiger-corrected} is stated with fixed values of the
stabilization coefficient $\beta$ and the geometry-adjusted similarity
parameter $\delta_\omega$.  Sharp global values are rarely available and
the analytic bounds may be conservative.  Adaptive \TIGER{} retains the
local and mirror updates of the fixed-parameter method, but chooses these
two quantities by backtracking.  This section gives the acceptance
conditions and their convergence consequences.  The complete procedure is
stated as Algorithm~\ref{alg:adaptive-tiger} in the appendix.

At iteration $k$, let $\beta>0$ and
$\widehat\delta_\omega>0$ denote the current trial values.  The acceleration
weight and extrapolation point are recomputed as
\[
 \tau_k=
 \left(
  \max\left\{1,\sqrt{\widehat\delta_\omega/\mu}\right\}
 \right)^{-1},
 \qquad
 y_k=\tau_kx_k+(1-\tau_k)v_k.
\]
 Using the current $\beta$, the
local routine seeks $v_{k+1}\in\X$ and
\[
 e_k\in
 \nabla p(y_k)+\partial q(v_{k+1})
 +\beta\bigl(\nabla\omega(v_{k+1})-\nabla\omega(y_k)\bigr),
 \qquad
 \dnorm{e_k}^{2}\le\xi_k,
\]
where $\xi_k$ is the prescribed accuracy of the local solve.  After the
synchronized query at $v_{k+1}$, define
\begin{equation}
 h_k:=
 \beta\bigl(\nabla\omega(v_{k+1})-\nabla\omega(y_k)\bigr)
 -\bigl(\nabla p(v_{k+1})-\nabla p(y_k)\bigr),
 \qquad
 S_k:=\ip{h_k}{v_{k+1}-y_k}.
 \label{eq:adaptive-secant-test}
\end{equation}
Since $h_k=\nabla R_\beta(v_{k+1})-\nabla R_\beta(y_k)$, the condition
$S_k\ge0$ checks monotonicity of the stabilized model error along the
computed secant.  If the local routine does not reach the prescribed
residual accuracy, or if the resulting candidate has $S_k<0$, the method
doubles $\beta$ and solves the local model again.  The query at $y_k$ is
unchanged and can be reused.

Once $S_k\ge0$, the vector
$g_k:=e_k-h_k$ belongs to
$\nabla p(v_{k+1})+\partial q(v_{k+1})$.  The method therefore uses the same
mirror update as fixed-parameter \TIGER{}:
\begin{equation}
 x_{k+1}:=
 \argmin_{u\in\X}
 \left\{
  \frac{\tau_k}{\mu}\ip{g_k}{u}
  +\tau_kD_\omega(u,v_{k+1})
  +D_\omega(u,x_k)
 \right\}.
 \label{eq:adaptive-outer-trial}
\end{equation}
Set
$ C_k:=\ip{h_k}{x_{k+1}-x_k},$
and
$ V_k:=D_\omega(x_{k+1},x_k).$
The iteration is accepted when
\begin{equation}
 \tau_k C_k\le S_k+\mu V_k.
 \label{eq:adaptive-signed-test}
\end{equation}
This inequality is the model-error absorption condition used in the convergence
proof.  In particular, it holds automatically when $C_k\le0$.  If it is
violated, $\widehat\delta_\omega$ is doubled and the candidate is recomputed,
including $\tau_k$ and $y_k$.  The two backtracking rules consequently
estimate different features of the problem: $\beta$ makes the local model
error monotone along the observed secant, whereas
$\widehat\delta_\omega$ determines how aggressively that error can be
absorbed by the accelerated mirror step.

The next lemma shows that these computable conditions are sufficient for
the same one-step estimate used in the fixed-parameter analysis.

\begin{lemma}[Accepted adaptive step]
\label{lem:adaptive-one-step}
Suppose Assumption~\ref{ass:outer} holds.  Every accepted iteration of
Adaptive \TIGER{} satisfies
\[
 \mu(1+\tau_k)D_\omega(x^\star,x_{k+1})
 +\Phi(v_{k+1})-\Phi^\star
 \le
 \mu D_\omega(x^\star,x_k)
 +(1-\tau_k)\bigl(\Phi(v_k)-\Phi^\star\bigr)
 +\frac{\xi_k}{2\mu}.
\]
\end{lemma}

The lemma depends only on the residual bound, $S_k\ge0$, and
\eqref{eq:adaptive-signed-test}.  A global similarity constant is needed
only to show that the backtracking terminates and that the accepted
weights do not become arbitrarily small.

\begin{theorem}[Convergence with adaptive parameters]
\label{thm:adaptive-rate}
Suppose the conditions of Lemma~\ref{lem:adaptive-one-step} hold and the
accepted weights satisfy $\tau_k\ge\underline\tau>0$.  Let
\[
 \overline B_0
 \ge
 \Phi(v_0)-\Phi^\star
 +\mu(1+\underline\tau)D_\omega(x^\star,x_0),
 \qquad
 \rho:=1-\frac{\underline\tau}{4},
\]
and choose the local accuracies so that
\[
 \frac{\xi_k}{2\mu}
 \le
 \frac{\underline\tau\,\overline B_0}{8}\rho^{k+1}.
\]
Then, for every $K\ge0$,
\[
 \Phi(v_K)-\Phi^\star
 +\mu(1+\underline\tau)D_\omega(x^\star,x_K)
 \le
 \overline B_0\rho^K.
\]
If $x_0=v_0$, the computable value
$\overline B_0=4\dnorm{g_0}^2/\mu$ from
Proposition~\ref{prop:initial-bound} is admissible.
\end{theorem}

The communication bound requires both backtracking loops to terminate after
finitely many trials.  We use the following uniform bounds.

\begin{assumption}[Uniform backtracking bounds]
\label{ass:adaptive-backtracking}
There exist finite constants
$\overline\beta>0$ and $\overline\delta_\omega\ge0$ with the following
properties.  For every local subproblem generated by the adaptive procedure,
the local routine returns
$ z\in\X,$
and $ e\in\partial\mathcal M_y(z)$.  Whenever
$\beta\ge\overline\beta$, the returned pair satisfies both
$\dnorm{e}^{2}\le\xi_k$ and the secant condition $S\ge0$.
Moreover, every queried candidate with $S\ge0$ satisfies
\[
 \left|\ip{h}{a-b}\right|^2
 \le
 2\overline\delta_\omega S D_\omega(a,b)
 \qquad
 \text{for all }a,b\in\X,
\]
where $h$ and $S$ are defined in
\eqref{eq:adaptive-secant-test}.
\end{assumption}
For a fixed stabilization coefficient,
Assumption~\ref{ass:correction} and
Lemma~\ref{lem:model-error-four-point} imply the second part with
$\overline\delta_\omega=\delta_\omega$.  When $\beta$ varies, uniformity
over the queried values is an additional condition.

\begin{corollary}[Communication complexity of Adaptive \TIGER{}]
\label{cor:adaptive-communication}
Suppose Assumptions~\ref{ass:outer} and
\ref{ass:adaptive-backtracking} hold.  For initial estimates
$\beta_0,\widehat\delta_{\omega,0}>0$, define
\[
 L_\beta:=
 \max\left\{
  0,
  \left\lceil\log_2\left(\frac{\overline\beta}{\beta_0}\right)\right\rceil
 \right\},
 \qquad
 L_\delta:=
 \max\left\{
  0,
  \left\lceil
   \log_2\left(
    \frac{\overline\delta_\omega}{\widehat\delta_{\omega,0}}
   \right)
  \right\rceil
 \right\},
\]
with $L_\delta=0$ when $\overline\delta_\omega=0$, and let
$ \Delta_{\rm ad}:=
 \max\left\{
  \widehat\delta_{\omega,0},2\overline\delta_\omega
 \right\}.$
If accepted estimates are carried monotonically between iterations and the
residual tolerances satisfy the condition of
Theorem~\ref{thm:adaptive-rate}, then Adaptive
\TIGER{} returns $v_K$ with
$\Phi(v_K)-\Phi^\star\le\varepsilon$ after
\[
 O\left(
  \left(1+\sqrt{\frac{\Delta_{\rm ad}}{\mu}}\right)
  \log\left(1+\frac{\overline B_0}{\varepsilon}\right)
  +L_\beta+L_\delta
 \right)
\]
synchronized queries.
\end{corollary}

Thus Adaptive \TIGER{} retains the asymptotic communication complexity of
the fixed-parameter method.  The additional logarithmic terms account for
the backtracking required when the initial parameter estimates are too small.

\section{Special cases and applications}
\label{sec:geometry}

This section verifies Assumption~\ref{ass:correction} in quadratic Hilbert and
entropy geometries.  The first two subsections state explicit cocoercivity
bounds and give proof sketches; complete proofs and implementation formulas
are in the appendix.  The final subsection applies the entropy result to
fixed-support entropic Wasserstein barycenters.

\subsection{Fixed Hilbert specialization}

Let $A=A^\top\succ0$ and take the quadratic distance-generating function
\[
 \omega(x)=\frac12\Anorm{x}^2,
 \qquad
 D_\omega(y,x)=\frac12\Anorm{y-x}^2.
\]

\begin{proposition}[Model-error bound under Hessian similarity]
\label{prop:hilbert-cocoercivity}
Suppose that condition~\eqref{eq:reference-hilbert-similarity} holds with
$\delta\ge0$.
With $\beta=2\delta$, the function $R_\beta(x)=2\delta\omega(x)-p(x)$ is
convex and $3\delta$-smooth with respect to the $A$-norm; when $\delta>0$,
it is $\delta$-strongly convex.
Moreover, Assumption~\ref{ass:correction} holds with
\begin{equation*}
 \delta_{\omega}=3\delta.
\end{equation*}
\end{proposition}

\noindent\textit{Proof sketch.}
The similarity bound and \(\beta=2\delta\) give
$\delta A\preceq\nabla^2R_\beta\preceq3\delta A$.  Averaging this Hessian
along the segment from \(x\) to \(y\) produces a positive-semidefinite
operator \(B\) with \(B^2\preceq3\delta B\).  Hence
\[
 \Adnorm{\nabla R_\beta(y)-\nabla R_\beta(x)}^2
 \le 3\delta\,
 \ip{\nabla R_\beta(y)-\nabla R_\beta(x)}{y-x}.
\]
This is Assumption~\ref{ass:correction} with
$\delta_{\omega}=3\delta$.  The full operator calculation is given in
Appendix~\ref{app:proof-hilbert-cocoercivity}.

Thus, in fixed Hilbert geometry, the geometry-adjusted similarity parameter
is controlled by the server-to-global Hessian-similarity scale.

\subsection{Entropy specialization on the simplex}

Let $d\ge2$ and $0<\alpha<1/d$, and define
\begin{equation*}
 \Delta_d:=\{x\in\R^d:x_j\ge0,\ \mathbf1^\top x=1\},
 \qquad
 \Delta_d^\alpha:=\{x\in\Delta_d:x_j\ge\alpha,
 \ j=1,\ldots,d\}.
\end{equation*}
We take
\begin{equation*}
 \omega(x)=\sum_{i=1}^d x_i\log x_i,
 \qquad
 \norm{v}=\norm{v}_1
 \quad\text{whenever }\mathbf1^\top v=0.
\end{equation*}
The entropy Bregman divergence satisfies
\begin{equation}
 D_\omega(y,x)=\KL(y\Vert x)=\sum_i y_i\log\frac{y_i}{x_i}
 \ge\frac12\norm{y-x}_1^2.
 \label{eq:pinsker}
\end{equation}
For the simplex, the feasible directions are the vectors
$v\in\R^d$ satisfying $\mathbf1^\top v=0$.  We use
$\norm{v}=\norm{v}_1$ on these directions, and all gradients and Hessians
below are understood along them.

\begin{proposition}[Model-error bound in entropy geometry]
\label{prop:entropy-cocoercivity}
Let $R$ be twice continuously differentiable along the simplex on a
neighborhood of $\Delta_d^\alpha$.  Suppose that, for every
$x\in\Delta_d^\alpha$ and $v\in\R^d$ such that $\mathbf1^\top v=0$,
\begin{equation*}
 0\le v^\top\nabla^2R(x)v
 \le\delta_{\omega}\norm{v}_1^2.
\end{equation*}
Then $R$ is convex on $\Delta_d^\alpha$ and, for every
$x,y\in\Delta_d^\alpha$,
\begin{equation*}
 \dnorm{\nabla R(y)-\nabla R(x)}^2
 \le \delta_{\omega}
 \ip{\nabla R(y)-\nabla R(x)}{y-x}.
\end{equation*}
Consequently, when $R=R_\beta$, Assumption~\ref{ass:correction} holds with
parameter $\delta_{\omega}$ and $D_\omega=\KL$.
\end{proposition}

\noindent\textit{Proof sketch.}
Averaging $\nabla^2R$ along the segment from $x$ to $y$ and applying
Cauchy--Schwarz for the resulting positive-semidefinite matrix gives, for
every tangent vector $v$ with $\norm{v}_1\le1$,
\[
 \left|\ip{\nabla R(y)-\nabla R(x)}{v}\right|^2
 \le \delta_{\omega}
 \ip{\nabla R(y)-\nabla R(x)}{y-x}.
\]
Taking the supremum over such $v$ proves the asserted cocoercivity bound, and
the lower Hessian bound gives convexity.  For $R=R_\beta$,
Lemma~\ref{lem:model-error-four-point} and
Pinsker's inequality~\eqref{eq:pinsker} then supply the four-point bound.
The complete proof is given in
Appendix~\ref{app:proof-entropy-cocoercivity}.

For $R=R_\beta$, the quadratic form in
Proposition~\ref{prop:entropy-cocoercivity} is
\[
 v^\top\nabla^2R_\beta(x)v
 =\beta\sum_i\frac{v_i^2}{x_i}-v^\top\nabla^2p(x)v.
\]
Thus the geometry-adjusted similarity parameter is obtained from an intrinsic
entropy-coordinate Hessian bound.  The quotient-dual identity and the constrained
two-center update are recorded in
Appendix~\ref{app:entropy-implementation}.  Relative strong convexity remains
a separate requirement. The next subsection verifies it for entropically
regularized Wasserstein barycenters.

\subsection{Entropic Wasserstein barycenters}
\label{sec:barycenter-specialization}

We use the fixed-support barycenter problem
\citep{agueh2011barycenters} as a concrete application of the
non-Euclidean theory.  Let
$a_1,\ldots,a_n\in\Delta_d$ have strictly positive coordinates, and let
$C\in\R_+^{d\times d}$ be a finite symmetric ground-cost matrix.  For
$x,a\in\Delta_d$, define the transportation polytope
\begin{equation*}
 \mathcal U(x,a):=
 \left\{
 \pi\in\R_+^{d\times d}:
 \pi\mathbf1=x,\ 
 \pi^\top\mathbf1=a
 \right\}
\end{equation*}
and the entropy-regularized optimal transport cost
\citep{cuturi2013sinkhorn,cuturi2014fast}
\begin{equation*}
 W_\gamma(x,a):=
 \min_{\pi\in\mathcal U(x,a)}
 \left\{
  \langle C,\pi\rangle
  +\gamma\sum_{j,\ell=1}^d
  \pi_{j\ell}\log\pi_{j\ell}
 \right\},
\end{equation*}
where $\gamma>0$ and $0\log0:=0$.
Agent $i$ stores the histogram $a_i$ and its corresponding transport loss.
We consider the barycenter problem on the truncated simplex:
\begin{equation}
 \min_{x\in\Delta_d^\alpha}
 F(x):=
 \frac1n\sum_{i=1}^n W_\gamma(x,a_i),
 \label{eq:truncated-barycenter}
\end{equation}
where  $0<\alpha<\frac1d.$
Restricting the problem to $\Delta_d^\alpha$ keeps the logarithmic
operations well defined and provides the uniform curvature bounds used
below.  In the notation of Section~\ref{sec:setting}, we set
$ \X=\Delta_d^\alpha,$
$ r\equiv0,$
$ f_i(x)=W_\gamma(x,a_i),$
$ p(x)=F(x)-f_1(x),$ and
$ q(x)=f_1(x).$
We use the negative-entropy distance-generating function
$ \omega(x)=\sum_{j=1}^d x_j\log x_j,$
and
$ D_\omega(y,x)=\KL(y\Vert x).$

Define the curvature discrepancy
\begin{equation}
 M_\alpha:=
 \max_{z\in\Delta_d^\alpha}
 \max_{\substack{
 u\in\R^d,\ \mathbf1^\top u=0\\
 \norm{u}_1=1}}
 \left|
 u^\top\nabla^2p(z)u
 \right|.
 \label{eq:barycenter-Malpha}
\end{equation}

\begin{lemma}[Barycenter relative geometry and model-error bound]
\label{lem:barycenter-relative-geometry}
For every $a\in\Delta_d$ with strictly positive coordinates, the function
$W_\gamma(\cdot,a)-\gamma\omega$ is convex on
$\Delta_d^\alpha$.  Moreover, $M_\alpha<\infty$.
Consequently, Assumption~\ref{ass:outer} holds for
problem~\eqref{eq:truncated-barycenter} with
$ \mu=\gamma,$
and Assumption~\ref{ass:correction} holds with
\begin{equation*}
 \beta=M_\alpha,
 \qquad
 \delta_\omega=
 \left(1+\frac1\alpha\right)M_\alpha.
\end{equation*}
\end{lemma}
For initialization at the uniform histogram, define
\begin{equation}
  u:=\frac1d\mathbf1,
 \quad
 g_0:=
 \frac1n\sum_{i=1}^n\nabla_xW_\gamma(u,a_i), \quad \overline B_0:=
 \frac1\gamma
 \left(
  \max_{1\le j\le d}(g_0)_j
  -
  \min_{1\le j\le d}(g_0)_j
 \right)^2.
 \label{eq:barycenter-initial-scale}
\end{equation}
\begin{theorem}[Fixed-parameter \TIGER{} for entropic Wasserstein barycenters]
\label{thm:barycenter} 
Run Algorithm~\ref{alg:tiger-corrected} with
$x_0=v_0=u$ and use
$ \mu=\gamma,$
$ \beta=M_\alpha,$
$ \delta_\omega=
 \left(1+\frac1\alpha\right)M_\alpha,$
and residual tolerances set according
to~\eqref{eq:residual-schedule}.  For every $\varepsilon>0$, after
\begin{equation*}
 K=
 O\left(
 \left(
  1+
  \sqrt{
   \frac{(1+\alpha^{-1})M_\alpha}{\gamma}
  }
 \right)
 \log\left(
  1+\frac{\overline B_0}{\varepsilon}
 \right)
 \right)
\end{equation*}
outer iterations, and hence the same order of synchronized communication
rounds, the output $v_K$ satisfies
$ \frac1n\sum\limits_{i=1}^n W_\gamma(v_K,a_i)
 -
 \min\limits_{x\in\Delta_d^\alpha}
 \frac1n\sum\limits_{i=1}^n W_\gamma(x,a_i)
 \le\varepsilon.$
\end{theorem}

\begin{proof}
Lemma~\ref{lem:barycenter-relative-geometry} verifies the assumptions of
Theorem~\ref{thm:communication} with
$ \mu=\gamma,$
$ \beta=M_\alpha,$
and
$ \delta_\omega=
 \left(1+\frac1\alpha\right)M_\alpha.$
At the uniform initialization, the initial averaged transport gradient is
$g_0$.  The computable initialization bound and the simplex dual-norm
identity give
\begin{align*}
 &\gamma(1+\tau)\KL(x^\star\Vert u)
 +F(u)-F(x^\star)
\le
 \frac{2(1+\tau)}{\gamma}\dnorm{g_0}^{\,2}
 =
 \frac{1+\tau}{2\gamma}
 \left(
  \max_j(g_0)_j-\min_j(g_0)_j
 \right)^2
 \le\overline B_0,
\end{align*}
where $x^\star$ is any minimizer of
problem~\eqref{eq:truncated-barycenter}.  Therefore
Theorem~\ref{thm:communication} applies with the value
of $\overline B_0$ in~\eqref{eq:barycenter-initial-scale} and yields
the statement.
\end{proof}

The theorem gives explicit parameters for \TIGER{}.  On the same
barycenter model, Adaptive \TIGER{} can instead select them using the secant
and model-error absorption conditions of Section~\ref{sec:adaptive-tiger}; the full
procedure is given in Algorithm~\ref{alg:adaptive-tiger} in the appendix.
Corollary~\ref{cor:adaptive-communication} applies whenever
Assumption~\ref{ass:adaptive-backtracking} holds over the stabilization
values queried during the run.  We keep this condition separate from
Lemma~\ref{lem:barycenter-relative-geometry}, which verifies the fixed
parameter assumptions but does not by itself certify all varying
stabilization values.

\section{Numerical experiments}
\label{sec:experiments}

We evaluate Adaptive \TIGER{} on entropic Wasserstein barycenter problems
using Fashion-MNIST images.  We form eight
distributed problems: four \textit{within-class} tasks, in which all ten
images belong to the same clothing category, and four \textit{cross-class}
tasks, in which each of the ten Fashion-MNIST categories is represented
once.  Each image is assigned to a separate agent.

\paragraph{Fashion-MNIST dataset.}
The original $28\times28$ images are reduced to $7\times7$ by block
averaging, normalized as probability histograms, and mixed with uniform
mass $\rho=10^{-3}$.  Thus $d=49$ and the feasible set is the truncated
simplex $\Delta_d^\alpha$ with $\alpha=\rho/d$.  We use the normalized
squared Euclidean cost on the two-dimensional pixel grid, set
$\gamma=0.005$, and initialize every method at the uniform histogram.

\paragraph{Methods.}
We compare Adaptive \TIGER{} with SPAG
\citep{hendrikx2020statistically}, a deterministic PAUS specialization
\citep{beznosikov2024bregman}, entropy mirror descent
\citep{beck2003mirror}, accelerated projected gradient applied to the full
objective \citep{nesterov2004intro}, and Euclidean projected gradient
descent.  These methods can be applied directly to the truncated-simplex
problem.  Methods whose published formulations do not directly accommodate
the constrained composite objective are excluded.

A single parameter configuration of each method is used across all eight
tasks.  Adaptive \TIGER{} selects its stabilization coefficient and
similarity estimate using the secant and model-error absorption conditions in
Section~\ref{sec:adaptive-tiger}; trial values are initialized from recently
accepted estimates.  The adaptive configuration was selected using these
same tasks.  All synchronized algorithmic queries, including those from
rejected adaptive or line-search trials, are counted.  Setup and server-local
computation are not included in the query count.

\paragraph{Evaluation criterion.}
Since the exact optimal value is unavailable, for each problem we compute a
common numerical reference value $F_{\mathrm{ref}}$ using a high-accuracy
centralized solve.  We report the relative objective gap
\[
 \operatorname{relgap}(x)
 :=
 \frac{F(x)-F_{\mathrm{ref}}}
      {F(x_0)-F_{\mathrm{ref}}}.
\]
The same reference value  and initialization (for details see  Appendix~\ref{app:fashion-protocol}) are used for all methods.

\begin{figure}[ht]
 \centering
 \includegraphics[width=\textwidth]
 {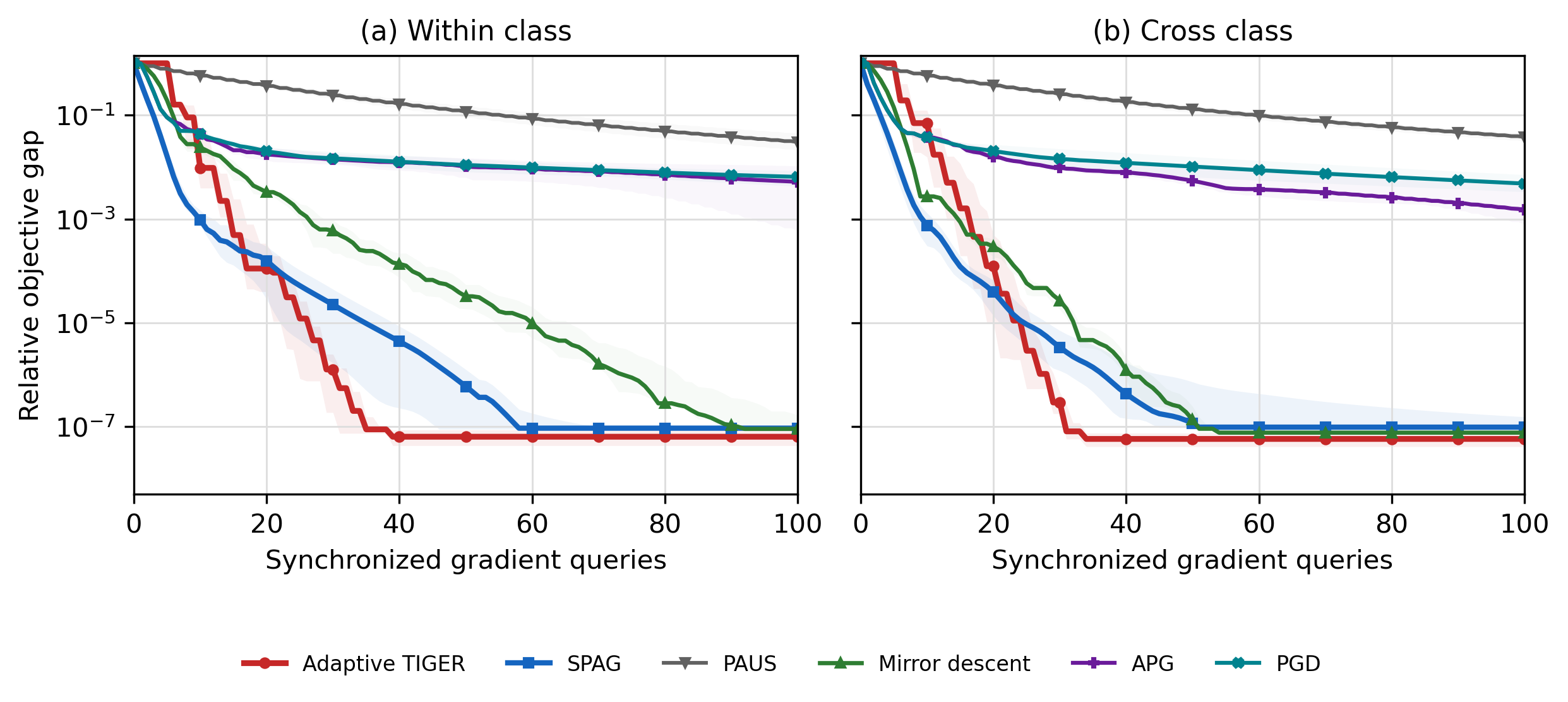}
 \caption{Fashion-MNIST barycenter experiments.  Curves show the median
 best-so-far relative objective gap, and shaded regions show the
 interquartile range over four within-class problems (left) and four
 cross-class problems (right).  The horizontal axis counts synchronized
 queries and is shown on a linear scale through $100$.}
 \label{fig:fashion-all-methods}
\end{figure}

Figure~\ref{fig:fashion-all-methods} shows that Adaptive \TIGER{} reaches
small relative objective gaps with fewer synchronized queries than the other
displayed methods in both task families.  The PAUS results are shown for
reference; its local solver did not attain the prescribed
stationarity tolerance in three of the eight problems.

\section{Conclusion}
\label{sec:conclusion}

The analysis of \TIGER{} identifies the stabilized local-model error as
the quantity that links function similarity to acceleration in a chosen
Bregman geometry.  Controlling this error preserves the separation between
communication and local computation on constrained domains, while a
computable residual allows the local subproblems to be solved inexactly.
Adaptive \TIGER{} uses the same convergence estimate to define acceptance
conditions for its stabilization and acceleration parameters.
The Wasserstein barycenter experiments illustrate the practical use of
this adaptive procedure in entropy geometry.

\section*{Acknowledgments}

The research was supported by Russian Science Foundation (project No. 23-11-00229-{\fontencoding{T2A}\fontfamily{cmr}\selectfont П}),
\url{https://rscf.ru/en/project/23-11-00229/}.

\begin{samepage}
\paragraph{Declaration on the use of generative AI.}
Generative AI tools were used for language editing and limited assistance
with code organization.  They were not used to formulate the research
questions or mathematical results.  The authors reviewed all AI-assisted
edits and take full responsibility for the content of the manuscript.
\par
\end{samepage}

\newpage

\par\bigskip
\phantomsection
\section*{Appendix}
\addcontentsline{toc}{section}{Appendix}
\appendix

\section{Proofs of the convergence results}
\label{app:convergence-proofs}

In the proofs only, we abbreviate the nonnegative monotonicity quantity
\begin{equation*}
 S_\beta(y,x)
 :=\ip{\nabla R_\beta(y)-\nabla R_\beta(x)}{y-x}\ge0,
\end{equation*}
where nonnegativity follows from the convexity of $R_\beta$ in
Assumption~\ref{ass:correction}.
We also use the Lyapunov quantity $\mathcal L_k$ defined in
\eqref{eq:lyapunov-quantity}.  It is a shorthand for the objective gap and
the Bregman term and is not used by the algorithm or its stopping rule.

\subsection{One-step estimate and Lyapunov recursion}
\label{app:proof-one-step}

\begin{lemma}[One-step estimate and Lyapunov recursion]
\label{lem:one-step}
Under Assumptions~\ref{ass:outer} and~\ref{ass:correction}, every iteration
satisfies
\begin{align*}
 \mathcal L_{k+1}
&\le
 \mu D_\omega(x^\star,x_k)
 +(1-\tau)\bigl(\Phi(v_k)-\Phi^\star\bigr)
 +\frac{\dnorm{e_k}^2}{2\mu}
\le
 \frac{\mathcal L_k}{1+\tau}
 +\frac{\xi_k}{2\mu}.
\end{align*}
\end{lemma}

\begin{proof}[Proof of Lemma~\ref{lem:one-step}]
Suppress the iteration index and write
\[
 x=x_k,\quad v=v_k,\quad y=y_k,\quad z=v_{k+1},\quad
 x^+=x_{k+1},\quad g=g_k,\quad e=e_k,
\]
and define the auxiliary correction vector
\[
 c:=\nabla p(z)-\nabla p(y)
 -\beta\bigl(\nabla\omega(z)-\nabla\omega(y)\bigr).
\]
By the residual inclusion, choose $s\in\partial q(z)$ such that
$e=\nabla p(y)+s+\beta(\nabla\omega(z)-\nabla\omega(y))$.
Then $g=c+e=\nabla p(z)+s$.

The first-order variational inequality for the outer subproblem in
Algorithm~\ref{alg:tiger-corrected} is
\[
 \ip{
  \frac{\tau}{\mu}g+\tau(\nabla\omega(x^+)-\nabla\omega(z))
  +\nabla\omega(x^+)-\nabla\omega(x)
 }{a-x^+}\ge0
\]
for every admissible $a\in\X$.  The Bregman three-point identity
\begin{equation*}
 \ip{\nabla\omega(a)-\nabla\omega(b)}{d-a}
 =D_\omega(d,b)-D_\omega(d,a)-D_\omega(a,b)
\end{equation*}
and the choice $a=x^\star$ in the variational inequality give
\begin{align}
 \frac{\tau}{\mu}\ip{g}{x^+-x^\star}
 +(1+\tau)D_\omega(x^\star,x^+)
 &\le D_\omega(x^\star,x)+\tau D_\omega(x^\star,z)
 -D_\omega(x^+,x)-\tau D_\omega(x^+,z).
 \label{eq:mirror-three-point}
\end{align}

Apply Assumption~\ref{ass:outer} at $z$ with the
composite subgradient $g=\nabla p(z)+s$.  With comparison points $x^\star$
and $v$, respectively,
\begin{equation*}
 \Phi(z)-\Phi^\star\le\ip{g}{z-x^\star}-\mu D_\omega(x^\star,z),
 \qquad
 \Phi(z)-\Phi(v)\le\ip{g}{z-v}-\mu D_\omega(v,z).
\end{equation*}
Multiply the first inequality by $\tau/\mu$, the second by
$(1-\tau)/\mu$, and add both to
\eqref{eq:mirror-three-point}.  Since
\[
 \tau(\Phi(z)-\Phi^\star)+(1-\tau)(\Phi(z)-\Phi(v))
 =\Phi(z)-\Phi^\star-(1-\tau)(\Phi(v)-\Phi^\star),
\]
we obtain
\begin{align}
 (1+\tau)D_\omega(x^\star,x^+)
 +\frac{\Phi(z)-\Phi^\star}{\mu}
&\le D_\omega(x^\star,x)
 +\frac{1-\tau}{\mu}\bigl(\Phi(v)-\Phi^\star\bigr)
  \notag\\*
 &-D_\omega(x^+,x)
-\tau D_\omega(x^+,z)-(1-\tau)D_\omega(v,z)
 +\frac1\mu\ip{g}{w},
 \label{eq:preliminary-one-step}
\end{align}
where the extrapolation identity gives
\begin{equation*}
 w:=\tau(z-x^+)+(1-\tau)(z-v)
 =z-y+\tau(x-x^+).
\end{equation*}

Moreover,
\[
 \nabla R_\beta(z)-\nabla R_\beta(y)=-c,
 \qquad
 \ip{c}{z-y}=-S_\beta(z,y).
\]
Consequently,
\begin{equation}
 \ip{g}{w}
 =-S_\beta(z,y)+\tau\ip{c}{x-x^+}+\ip{e}{w}.
 \label{eq:g-w-decomposition}
\end{equation}

Set $S=S_\beta(z,y)$ and $V=D_\omega(x^+,x)$.  Applying
Lemma~\ref{lem:model-error-four-point} with $(x,y,a,b)=(y,z,x^+,x)$ yields
\[
 |\ip{c}{x-x^+}|^2\le2\delta_{\omega} SV.
\]
Young's inequality and $\tau^2\delta_{\omega}\le\mu$ give
\begin{align*}
 \frac{\tau}{\mu}\ip{c}{x-x^+}
 &\le\frac{S}{2\mu}+\frac{\tau^2\delta_{\omega}}{\mu}V,\\
 -\frac S\mu-V+\frac{\tau}{\mu}\ip{c}{x-x^+}
 &\le-\frac{S}{2\mu}
 -\left(1-\frac{\tau^2\delta_{\omega}}{\mu}\right)V\le0.
\end{align*}

It remains to absorb the local residual.  By~\eqref{eq:omega-strong} and convexity
of the squared norm,
\begin{align*}
 \tau D_\omega(x^+,z)+(1-\tau)D_\omega(v,z)
 &\ge\frac12\left[
   \tau\norm{x^+-z}^2+(1-\tau)\norm{v-z}^2\right]\\
 &\ge\frac12\norm{
   \tau(x^+-z)+(1-\tau)(v-z)}^2
 =\frac12\norm{w}^2.
\end{align*}
Hence the Fenchel--Young inequality gives
\begin{align}
 \frac1\mu\ip{e}{w}
 -\tau D_\omega(x^+,z)-(1-\tau)D_\omega(v,z)
 \le\frac{\dnorm{e}}{\mu}\norm{w}-\frac12\norm{w}^2 \le\frac{\dnorm{e}^2}{2\mu^2}.
 \label{eq:residual-absorption}
\end{align}

Substitute~\eqref{eq:g-w-decomposition} and~\eqref{eq:residual-absorption}
into~\eqref{eq:preliminary-one-step} and discard the nonpositive remainder:
\begin{equation*}
 (1+\tau)D_\omega(x^\star,x_{k+1})
 +\frac{\Phi(v_{k+1})-\Phi^\star}{\mu}\le D_\omega(x^\star,x_k)
 +\frac{1-\tau}{\mu}(\Phi(v_k)-\Phi^\star)
 +\frac{\dnorm{e_k}^2}{2\mu^2}.
\end{equation*}
Multiplying by $\mu$ gives the first inequality in Lemma~\ref{lem:one-step}.
Finally, $\Phi(v_k)-\Phi^\star\ge0$ and
$1-\tau\le(1+\tau)^{-1}$ give the second inequality.  This
proves the lemma.
\end{proof}

\subsection{Unrolling under the prescribed residual schedule}
\label{app:proof-rates}

\begin{proof}[Proof of Lemma~\ref{lem:rates} and~\eqref{eq:residual-rate}]
Repeated substitution in the second inequality of Lemma~\ref{lem:one-step} gives
\begin{equation*}
 \mathcal L_k
 \le (1+\tau)^{-k}\mathcal L_0
 +\frac1{2\mu}\sum_{j=0}^{k-1}
 (1+\tau)^{-(k-1-j)}\xi_j.
\end{equation*}
Expanding $\mathcal L_k$ and $\mathcal L_0$ proves
Lemma \ref{lem:rates}.  For the prescribed schedule, note that, for
$0<\tau\le1$,
\[
 \varrho:=1-\frac\tau4,
 \qquad
 \theta=\frac1{1+\tau}\le1-\frac\tau2,
 \qquad
 \theta+\frac{\tau\varrho}{8}
 \le1-\frac{3\tau}{8}\le1-\frac\tau4=\varrho.
\]
Suppose that, at iteration $k$,
$ \mathcal L_k\le \overline B_0\varrho^k.$
Then Lemma~\ref{lem:one-step} and~\eqref{eq:residual-schedule} give
\begin{align*}
 \mathcal L_{k+1}
 &\le\theta \overline B_0\varrho^k
   +\frac{\tau\overline B_0}{8}\varrho^{k+1}=\overline B_0\varrho^k
   \left(\theta+\frac{\tau\varrho}{8}\right)
 \le \overline B_0\varrho^{k+1}.
\end{align*}
The initial case follows from~\eqref{eq:initial-scale}, so induction
proves the bound on $\mathcal L_k$.  Expanding $\mathcal L_K$ gives
\eqref{eq:residual-rate}.  The scalar inequalities above hold for every
$0<\tau\le1$.
\end{proof}

\subsection{Explicit communication bound}
\label{app:proof-communication}

\begin{proof}[Proof of Theorem~\ref{thm:communication}]
Equation~\eqref{eq:residual-rate} gives
\[
\Phi(v_K)-\Phi^\star
\le
\overline B_0\left(1-\frac{\tau}{4}\right)^K.
\]
Since $0<\tau\le 1$ and
\[
-\log\left(1-\frac{\tau}{4}\right)\ge\frac{\tau}{4},
\]
we obtain
\[
\left(1-\frac{\tau}{4}\right)^K
\le
\exp\left(-\frac{\tau K}{4}\right).
\]
Therefore,
\[
\Phi(v_K)-\Phi^\star
\le
\overline B_0
\exp\left(-\frac{\tau K}{4}\right).
\]

Choose an integer
\[
 K\ge
 \left\lceil
 \frac4\tau
 \log\left(1+\frac{\overline B_0}{\varepsilon}\right)
 \right\rceil.
\]
Then $\tau K/4$ is at least
$\log(1+\overline B_0/\varepsilon)$, and hence
\[
\exp\left(-\frac{\tau K}{4}\right)
\le
\left(1+\frac{\overline B_0}{\varepsilon}\right)^{-1}.
\]
It follows that
\[
\Phi(v_K)-\Phi^\star
\le
\frac{\overline B_0}
{1+\overline B_0/\varepsilon}
=
\frac{\overline B_0\varepsilon}
{\overline B_0+\varepsilon}
\le
\varepsilon.
\]

Finally, by the definition of the acceleration parameter,
\[
\tau^{-1}
=
\max\left\{
1,\sqrt{\frac{\delta_{\omega}}{\mu}}
\right\}
=
O\left(
1+\sqrt{\frac{\delta_{\omega}}{\mu}}
\right).
\]
Substitution into the explicit bound gives the stated iteration complexity.
Each iteration uses two synchronized correction queries, so the total is
exactly $2K$ synchronized queries and has the same order.
\end{proof}

\subsection{A computable initialization bound}
\label{app:proof-initial-bound}

\begin{proposition}[Computable initialization bound]
\label{prop:initial-bound}
If $x_0=v_0\in\X$ and
$g_0\in\nabla p(x_0)+\partial q(x_0)$, then
\begin{align*}
 &\mu(1+\tau)D_\omega(x^\star,x_0)
 +\Phi(x_0)-\Phi^\star
\le\frac{2(1+\tau)}{\mu}\dnorm{g_0}^2
 \le\frac4\mu\dnorm{g_0}^2.
\end{align*}
Consequently, $\overline B_0=4\dnorm{g_0}^2/\mu$ is admissible
in~\eqref{eq:initial-scale}.
\end{proposition}

\begin{proof}[Proof of Proposition~\ref{prop:initial-bound}]
Put
$V_0=D_\omega(x^\star,x_0)$,
$G_0=\Phi(x_0)-\Phi^\star$, and
$h=\dnorm{g_0}/\mu$.  Relative strong convexity at $x_0$ gives
\[
 G_0+\mu V_0
 \le\ip{g_0}{x_0-x^\star}
 \le\dnorm{g_0}\sqrt{2V_0}.
\]
After division by $\mu$,
$G_0/\mu+V_0\le h\sqrt{2V_0}$.  Since $G_0\ge0$, this implies
$V_0\le2h^2$ and $G_0/\mu+V_0\le2h^2$.  Consequently,
\begin{align*}
 \mu(1+\tau)V_0+G_0
 &=\mu\left[\tau V_0+V_0+\frac{G_0}{\mu}\right]\le2\mu(1+\tau)h^2
 =\frac{2(1+\tau)}{\mu}\dnorm{g_0}^2.
\end{align*}
The final bound follows from $\tau\le1$.
\end{proof}

\section{Adaptive \TIGER{}: algorithm and proofs}
\label{app:adaptive-proofs}

The following pseudocode makes the adaptive procedure explicit.  Its
notation and order of operations parallel Algorithm~\ref{alg:tiger-corrected};
the only additions are the two backtracking loops.

\begin{algorithm}[H]
\caption{Adaptive \TIGER{}}
\label{alg:adaptive-tiger}
\KwIn{starting points $x_0=v_0\in\X$; $\mu>0$; initial estimates
$\beta_0,\widehat\delta_{\omega,0}>0$; number of iterations $K$;
residual bounds $\xi_k\ge0$.}

$\beta\leftarrow\beta_0$,
$\widehat\delta_\omega\leftarrow\widehat\delta_{\omega,0}$\;

\For{$k=0,\ldots,K-1$}{
 \Repeat{$\tau_k C_k\le S_k+\mu V_k$}{
  $\tau_k\leftarrow
   \left(\max\{1,\sqrt{\widehat\delta_\omega/\mu}\}\right)^{-1}$,
  $\quad y_k\leftarrow\tau_kx_k+(1-\tau_k)v_k$\;

  Synchronize at $y_k$ and set
  $\ell_k\leftarrow\nabla p(y_k)$\;

  \Repeat{$\dnorm{e_k}^{2}\le\xi_k$ and $S_k\ge0$}{
   Locally compute $v_{k+1}\in\X$ and
   $e_k\in\ell_k+\partial q(v_{k+1})
   +\beta\bigl(\nabla\omega(v_{k+1})-\nabla\omega(y_k)\bigr)$\;

   Synchronize at $v_{k+1}$ and set
   $h_k\leftarrow
   \beta\bigl(\nabla\omega(v_{k+1})-\nabla\omega(y_k)\bigr)
   -\bigl(\nabla p(v_{k+1})-\ell_k\bigr)$,
   $\quad S_k\leftarrow\ip{h_k}{v_{k+1}-y_k}$\;

   \If{$\dnorm{e_k}^{2}>\xi_k$ or $S_k<0$}{
    $\beta\leftarrow2\beta$\;
   }
  }

  $g_k\leftarrow e_k-h_k$\;

  $x_{k+1}\leftarrow
   \argmin\limits_{u\in\X}
   \left\{
    \frac{\tau_k}{\mu}\ip{g_k}{u}
    +\tau_kD_\omega(u,v_{k+1})
    +D_\omega(u,x_k)
   \right\}$\;

  $C_k\leftarrow\ip{h_k}{x_{k+1}-x_k}$,
  $\quad V_k\leftarrow D_\omega(x_{k+1},x_k)$\;

  \If{$\tau_k C_k>S_k+\mu V_k$}{
   $\widehat\delta_\omega\leftarrow2\widehat\delta_\omega$\;
  }
 }
}
\KwOut{$v_K$.}
\end{algorithm}

A failed stabilization trial requires a new synchronized query at the new
candidate $v_{k+1}$, while the query at $y_k$ is reused.  A failed
similarity trial changes $\tau_k$ and $y_k$ and therefore repeats both
synchronized queries.

\subsection{Accepted-step estimate}

\begin{proof}[Proof of Lemma~\ref{lem:adaptive-one-step}]
Suppress the iteration index and write
$x=x_k$, $v=v_k$, $y=y_k$, $z=v_{k+1}$,
$x^+=x_{k+1}$, $g=g_k$, $h=h_k$, and $e=e_k$.
The variational inequality for~\eqref{eq:adaptive-outer-trial}, with
comparison point $x^\star$, and the Bregman three-point identity give
\begin{align}
 \frac{\tau}{\mu}\ip{g}{x^+-x^\star}
 +(1+\tau)D_\omega(x^\star,x^+)
 \le{}&D_\omega(x^\star,x)+\tau D_\omega(x^\star,z)
 \notag\\
 &-D_\omega(x^+,x)-\tau D_\omega(x^+,z).
 \label{eq:adaptive-mirror-vi}
\end{align}
Since $g\in\nabla p(z)+\partial q(z)$, relative strong convexity at $z$,
with comparison points $x^\star$ and $v$, yields
\begin{align*}
 \Phi(z)-\Phi^\star
 &\le\ip{g}{z-x^\star}-\mu D_\omega(x^\star,z),\\
 \Phi(z)-\Phi(v)
 &\le\ip{g}{z-v}-\mu D_\omega(v,z).
\end{align*}
Multiply the first inequality by $\tau/\mu$, the second by
$(1-\tau)/\mu$, and add them to~\eqref{eq:adaptive-mirror-vi}.  Since
$y=\tau x+(1-\tau)v$, the remaining displacement is
\[
 w:=\tau(z-x^+)+(1-\tau)(z-v)
   =z-y+\tau(x-x^+).
\]
Consequently,
\begin{align}
 &(1+\tau)D_\omega(x^\star,x^+)
 +\frac{\Phi(z)-\Phi^\star}{\mu}
 \notag\\
 &\le D_\omega(x^\star,x)
 +\frac{1-\tau}{\mu}\bigl(\Phi(v)-\Phi^\star\bigr)
 -V-\tau D_\omega(x^+,z)-(1-\tau)D_\omega(v,z)
 +\frac1\mu\ip{g}{w},
 \label{eq:adaptive-preliminary}
\end{align}
where $V=D_\omega(x^+,x)$.  From
\eqref{eq:adaptive-secant-test}, $g=e-h$, and the definition of $w$,
\[
 \ip{g}{w}=-S+\tau C+\ip{e}{w}.
\]
The model-error absorption condition~\eqref{eq:adaptive-signed-test} gives
\[
 -V+\frac{-S+\tau C}{\mu}\le0.
\]
Moreover, the strong convexity of $\omega$ and the convexity of the squared
norm imply
\[
 \tau D_\omega(x^+,z)+(1-\tau)D_\omega(v,z)
 \ge\frac12\norm{w}^2.
\]
Fenchel--Young and $\dnorm{e}^2\le\xi_k$ therefore give
\[
 \frac1\mu\ip{e}{w}
 -\tau D_\omega(x^+,z)-(1-\tau)D_\omega(v,z)
 \le\frac{\xi_k}{2\mu^2}.
\]
Substitution into~\eqref{eq:adaptive-preliminary} and multiplication by
$\mu$ prove the claim.
\end{proof}

\subsection{Variable weights and residuals}

\begin{proof}[Proof of Theorem~\ref{thm:adaptive-rate}]
Define, for this proof,
\[
 \mathcal A_k:=\Phi(v_k)-\Phi^\star
 +\mu(1+\underline\tau)D_\omega(x^\star,x_k),
 \qquad
 a_k:=\frac{\xi_k}{2\mu}.
\]
Since $\tau_k\ge\underline\tau$ and the objective gap is nonnegative,
Lemma~\ref{lem:adaptive-one-step} and
$1-\underline\tau\le(1+\underline\tau)^{-1}$ give
\begin{equation}
 \mathcal A_{k+1}
 \le\frac{\mathcal A_k}{1+\underline\tau}+a_k.
 \label{eq:adaptive-fixed-recursion}
\end{equation}
Let $\vartheta=(1+\underline\tau)^{-1}$ and
$\rho=1-\underline\tau/4$.  For $0<\underline\tau\le1$,
\[
 \vartheta\le1-\frac{\underline\tau}{2},
 \qquad
 \vartheta+\frac{\underline\tau\rho}{8}
 \le1-\frac{3\underline\tau}{8}\le\rho.
\]
Starting from $\mathcal A_0\le\overline B_0$, induction in
\eqref{eq:adaptive-fixed-recursion} with the prescribed residual schedule
proves $\mathcal A_K\le\overline B_0\rho^K$.
Proposition~\ref{prop:initial-bound}, with $\tau$ replaced by
$\underline\tau$, gives the stated computable initial bound.
\end{proof}

If the residual tolerance is fixed at $\xi_k\le\overline\xi$, unrolling
\eqref{eq:adaptive-fixed-recursion} gives the numerical floor
\[
 \limsup_{K\to\infty}\bigl(\Phi(v_K)-\Phi^\star\bigr)
 \le
 \frac{1+\underline\tau}{\underline\tau}
 \frac{\overline\xi}{2\mu}.
\]

\subsection{Backtracking and communication complexity}

Assumption~\ref{ass:adaptive-backtracking} makes both backtracking loops finite.
The residual and secant conditions hold once $\beta$ reaches
$\overline\beta$.  For the
similarity condition, apply the four-point bound in
Assumption~\ref{ass:adaptive-backtracking} with $(a,b)=(x_{k+1},x_k)$:
\[
 C_k^2\le2\overline\delta_\omega S_kV_k.
\]
If $C_k\le0$, condition~\eqref{eq:adaptive-signed-test} holds immediately.
If $C_k>0$ and
$\widehat\delta_\omega\ge\overline\delta_\omega$, Young's inequality and
$\tau_k^2\overline\delta_\omega\le\mu$ give
\[
 \tau_kC_k
 \le\frac12S_k+\tau_k^2\overline\delta_\omega V_k
 \le S_k+\mu V_k.
\]

Let $K$ be the number of accepted iterations, $D$ the total number of
rejected similarity values, and $B$ the total number of candidate queries
rejected by the residual or secant condition.  There are $K+D$
queries at extrapolation points and $K+D+B$ queries at candidate points.
Thus
\begin{equation}
 Q_{\rm on}=2K+2D+B.
 \label{eq:adaptive-exact-ledger}
\end{equation}
Because the trial estimates are carried monotonically and doubled after a
rejection, $B\le L_\beta$ and $D\le L_\delta$.  Hence
\[
 Q_{\rm on}\le2K+L_\beta+2L_\delta.
\]
Every accepted similarity estimate is at most
$\Delta_{\rm ad}$, so
\[
 \tau_k\ge
 \underline\tau:=
 \frac{1}{
  \max\{1,\sqrt{\Delta_{\rm ad}/\mu}\}
 }.
\]
Finally, Theorem~\ref{thm:adaptive-rate} and
$\rho^K\le\exp(-K\underline\tau/4)$ show that
\[
 K\ge
 \left\lceil
  \frac4{\underline\tau}
  \log\left(1+\frac{\overline B_0}{\varepsilon}\right)
 \right\rceil
\]
is sufficient.  Combining this bound with
\eqref{eq:adaptive-exact-ledger} proves
Corollary~\ref{cor:adaptive-communication}.

\paragraph{Finite-precision acceptance.}
In computation, one may replace~\eqref{eq:adaptive-signed-test} by
$\tau_kC_k\le S_k+\mu V_k+\zeta_k$ with $\zeta_k\ge0$.
The accepted-step estimate then acquires the additive term $\zeta_k$.
Decreasing $\zeta_k$ geometrically preserves linear convergence, whereas a
fixed positive tolerance produces a corresponding numerical error floor.

\section{Finite local implementation}
\label{app:local-proofs}
\label{sec:local}

This section describes a local solver that meets the residual condition in
Algorithm~\ref{alg:tiger-corrected} in finitely many iterations.  Unlike the
communication theorem, this result requires a fixed quadratic metric
and a uniform warm-start bound.

Let $A=A^\top\succ0$ be fixed.  For $\lambda>0$, define the metric proximal
map
\begin{equation}
 \prox_{\lambda\mathcal R}(z)
 :=\argmin_x\left\{
  \lambda\mathcal R(x)+\frac12\norm{x-z}_A^2
 \right\}.
 \label{eq:metric-prox}
\end{equation}
Throughout this section, the smooth parts $G_k$ of the local objectives, and
$G$ in the auxiliary lemmas, satisfy the stated convexity and smoothness
bounds on the ambient space.  In particular, these bounds hold at all
extrapolated points of the local accelerated routine.  When the regularizer is $r$, we
extend it by $+\infty$ outside $\X$; its proximal map then enforces the
constraint.

\begin{proposition}[Complexity of residual certification]
\label{prop:local-work}
Let $\mathcal M_k=G_k+\mathcal R$, where $\mathcal R$ is proper, closed,
and convex, $G_k$ is convex and $L$-smooth in the $A$-metric, and
$\mathcal M_k$ is $m$-strongly convex and attains its minimum
$\mathcal M_k^\star$.  Suppose $L,m>0$ are uniform in $k$,
$\prox_{\mathcal R/L}$ is available, and the warm starts obey
\begin{equation*}
 \mathcal M_k(z_k^0)-\mathcal M_k^\star
 \le C_{\rm warm}\frac{\xi_k}{8L},
 \qquad C_{\rm warm}\ge1,
\end{equation*}
where $C_{\rm warm}$ is independent of $k$ and the requested residual level.
A restarted accelerated proximal-gradient method followed by one certificate step
returns $z_k$ and $e_k\in\partial\mathcal M_k(z_k)$ satisfying
$\norm{e_k}_{A^{-1}}^2\le\xi_k$ in
\begin{equation*}
 O\!\left((1+\sqrt{L/m})\log(2+C_{\rm warm})\right)
\end{equation*}
local gradient and proximal evaluations.
\end{proposition}

At iteration $k$ of Algorithm~\ref{alg:tiger-corrected}, take
$\mathcal M_k=\mathcal M_{y_k}$ and $v_{k+1}=z_k$.  A $K$-iteration run
therefore uses
$O\!\left(K(1+\sqrt{L/m})\log(2+C_{\rm warm})\right)$ local evaluations
and $2K$ synchronized correction queries.  The warm-start premise is a
separate implementation assumption and is not needed for
Theorem~\ref{thm:communication}.

In the quadratic specialization with $\beta=2\delta$, if $f_1$ is convex and
$L_1$-smooth and $\Phi$ is $\mu$-strongly convex, one may take
\begin{equation*}
 L=L_1+2\delta,
 \qquad
 m=\delta+\max\{\delta,\mu\}.
\end{equation*}
The proof is given in Proposition~\ref{prop:local-curvature} below.

For the auxiliary results, set
$\norm{x}=\norm{x}_A$, $\dnorm{u}=\norm{u}_{A^{-1}}$, and
$\mathsf Jx=Ax$; then
\[
 \ip{\mathsf Jx}{y}=x^\top Ay,
 \qquad
 \dnorm{\mathsf Jx}=\norm{x}.
\]
In this appendix, $\dom\mathcal R$ denotes the effective domain of
$\mathcal R$.
Using the metric proximal map in~\eqref{eq:metric-prox}, define
\[
 T_L(y):=\prox_{\mathcal R/L}\left(
 y-\frac1L\mathsf J^{-1}\nabla G(y)\right).
\]

\subsection{Accelerated proximal-gradient block}
\label{app:proof-apg-block}

\begin{lemma}[Accelerated proximal-gradient block]
\label{lem:apg-block}
Let $\mathcal M=G+\mathcal R$, where $G$ is convex, differentiable, and
$L$-smooth, $\mathcal R$ is proper, closed, and convex, and $\mathcal M$ is
$m$-strongly convex and attains its minimum at $x_{\mathcal M}^\star$.
Starting from
$\widetilde x_0\in\dom\mathcal R$, set $t_1=1$,
$\widetilde y_1=\widetilde x_0$, and, for $j=1,\ldots,B$, set
\begin{align*}
 \widetilde x_j&=T_L(\widetilde y_j),
 \\
 t_{j+1}&=\frac{1+\sqrt{1+4t_j^2}}2,
 \\
 \widetilde y_{j+1}&=\widetilde x_j
 +\frac{t_j-1}{t_{j+1}}(\widetilde x_j-\widetilde x_{j-1}).
\end{align*}
Then
\begin{equation*}
 \mathcal M(\widetilde x_B)-\mathcal M^\star
 \le\frac{2L}{(B+1)^2}\norm{\widetilde x_0-x_{\mathcal M}^\star}^2
 \le\frac{4L}{m(B+1)^2}
 \bigl(\mathcal M(\widetilde x_0)-\mathcal M^\star\bigr).
\end{equation*}
Thus
$ B=\left\lceil\sqrt{{8L}/{m}}\right\rceil$
reduces the gap by at least one half.
\end{lemma}

\begin{proof}[Proof of Lemma~\ref{lem:apg-block}]
Optimality of $T_L(y)$, $L$-smoothness of $G$, and convexity of $G$ imply,
for every $x\in\dom\mathcal R$,
\begin{equation}
 \mathcal M(T_L(y))
 \le\mathcal M(x)+\frac L2
 \left(\norm{x-y}^2-\norm{x-T_L(y)}^2\right).
 \label{eq:prox-three-point}
\end{equation}
Indeed, the proximal surrogate is $L$-strongly convex, so its value at its
minimizer plus $L\norm{x-T_L(y)}^2/2$ is at most its value at $x$.
The descent lemma bounds $\mathcal M(T_L(y))$ by the surrogate at the
minimizer, while convexity of $G$ bounds the surrogate at $x$ by
$\mathcal M(x)+L\norm{x-y}^2/2$.

Set
\[
 \Delta_j=\mathcal M(\widetilde x_j)-\mathcal M^\star,
 \qquad
 w_j=t_j\widetilde x_j-(t_j-1)\widetilde x_{j-1}-x_{\mathcal M}^\star.
\]
For $j=1,\ldots,B-1$, apply~\eqref{eq:prox-three-point} at iteration $j+1$
with comparison points
$\widetilde x_j$ and $x_{\mathcal M}^\star$, multiply the two inequalities by
$t_{j+1}(t_{j+1}-1)$ and $t_{j+1}$, and use
$t_{j+1}(t_{j+1}-1)=t_j^2$.  Expanding the squared norms gives
\begin{equation}
 \frac{2t_{j+1}^2}{L}\Delta_{j+1}+\norm{w_{j+1}}^2
 \le\frac{2t_j^2}{L}\Delta_j+\norm{w_j}^2.
 \label{eq:fista-potential}
\end{equation}
The required norm cancellation is the elementary identity
\begin{align*}
 t_{j+1}(t_{j+1}-1)
 \left(\norm{\widetilde x_j-\widetilde y_{j+1}}^2
 -\norm{\widetilde x_j-\widetilde x_{j+1}}^2\right)
 &+t_{j+1}
 \left(\norm{x_{\mathcal M}^\star-\widetilde y_{j+1}}^2
 -\norm{x_{\mathcal M}^\star-\widetilde x_{j+1}}^2\right)\\
 &=\norm{w_j}^2-\norm{w_{j+1}}^2.
\end{align*}
The initial application of~\eqref{eq:prox-three-point}, with
$\widetilde y_1=\widetilde x_0$
and $x=x_{\mathcal M}^\star$, gives explicitly
\[
 \frac2L\Delta_1+\norm{\widetilde x_1-x_{\mathcal M}^\star}^2
 \le\norm{\widetilde x_0-x_{\mathcal M}^\star}^2.
\]
Since $t_1=1$ and $w_1=\widetilde x_1-x_{\mathcal M}^\star$, this is the
first instance of the same estimate.
Telescoping
\eqref{eq:fista-potential} gives
\[
 \frac{2t_B^2}{L}\Delta_B\le\norm{\widetilde x_0-x_{\mathcal M}^\star}^2.
\]
The recurrence defining $t_j$ implies $t_B\ge(B+1)/2$, proving the first
inequality in Lemma~\ref{lem:apg-block}.  Strong convexity gives
$\norm{\widetilde x_0-x_{\mathcal M}^\star}^2
\le2(\mathcal M(\widetilde x_0)-\mathcal M^\star)/m$, which
proves the second.  The stated block length makes the last
coefficient at most $1/2$.
\end{proof}

\subsection{A computable proximal-gradient residual}
\label{app:proof-pg-certificate}

\begin{lemma}[Computable proximal-gradient residual]
\label{lem:pg-certificate}
Under the assumptions of Lemma~\ref{lem:apg-block}, let
$y\in\dom\mathcal R$ and $z=T_L(y)$.  Then
\[
 s:=-\nabla G(y)-L\mathsf J(z-y)\in\partial\mathcal R(z),
\]
and
\begin{equation*}
 e:=\nabla G(z)+s
 =\nabla G(z)-\nabla G(y)-L\mathsf J(z-y)
 \in\partial\mathcal M(z),
\end{equation*}
satisfies
\begin{equation*}
 \dnorm{e}^2\le8L\bigl(\mathcal M(y)-\mathcal M^\star\bigr).
\end{equation*}
\end{lemma}

\begin{proof}[Proof of Lemma~\ref{lem:pg-certificate}]
The inclusion for $s$ is the first-order condition of the proximal problem.
The convex smooth sum rule gives
 $\partial(G+\mathcal R)(z)=\nabla G(z)+\partial\mathcal R(z)$, proving
$e\in\partial\mathcal M(z)$.  Put $d=z-y$.  Convexity of $\mathcal R$ and
the descent lemma yield
\begin{align*}
 \mathcal R(y)&\ge\mathcal R(z)+\ip{s}{y-z}
 =\mathcal R(z)+\ip{\nabla G(y)}{d}+L\norm{d}^2,\\
 G(z)&\le G(y)+\ip{\nabla G(y)}{d}+\frac L2\norm{d}^2.
\end{align*}
Hence $\mathcal M(z)\le\mathcal M(y)-L\norm{z-y}^2/2$.  Moreover,
\[
 \dnorm{e}
 \le\dnorm{\nabla G(z)-\nabla G(y)}+L\norm{z-y}
 \le2L\norm{z-y}.
\]
Therefore
\[
 \dnorm{e}^2
 \le4L^2\norm{z-y}^2
 \le8L(\mathcal M(y)-\mathcal M(z))
 \le8L(\mathcal M(y)-\mathcal M^\star).
\]
\end{proof}

\subsection{Local complexity of residual certification}
\label{app:proof-local-work}

\begin{proof}[Proof of Proposition~\ref{prop:local-work}]
Set
$ B=\left\lceil\sqrt{{8L}/{m}}\right\rceil,$
and
$ J=\left\lceil\log_2 C_{\rm warm}\right\rceil.$
Every restarted block halves the current gap by
Lemma~\ref{lem:apg-block}.  If $\widehat y_k$ is the output after $J$ blocks,
then
\[
 \mathcal M_k(\widehat y_k)-\mathcal M_k^\star
 \le2^{-J}C_{\rm warm}\frac{\xi_k}{8L}
 \le\frac{\xi_k}{8L}.
\]
Lemma~\ref{lem:pg-certificate} now gives $\dnorm{e_k}^2\le\xi_k$.
Each FISTA iteration uses one gradient and one proximal evaluation.  The final
certificate step uses one proximal evaluation, its input gradient, and one
additional output gradient.  Thus, when the proposition is used at every
outer iteration of Algorithm~\ref{alg:tiger-corrected}, a $K$-iteration run
reporting $v_K$ uses at most $K(BJ+1)$ proximal evaluations and $K(BJ+2)$
gradient evaluations.
The asymptotic statement follows from
$B\le1+\sqrt{8L/m}$ and $J\le1+\log_2 C_{\rm warm}$.
\end{proof}

\subsection{Inherited local curvature}
\label{app:proof-local-curvature}

\begin{proposition}[Inherited local curvature]
\label{prop:local-curvature}
Let $A=A^\top\succ0$, suppose that
condition~\eqref{eq:reference-hilbert-similarity} holds with $\delta\ge0$, take
$\omega(x)=\norm{x}_A^2/2$, and set $\beta=2\delta$.
Suppose $f_1$ is convex and $L_1$-smooth in the $A$-metric, $r$ is proper,
closed, and convex, and $\Phi=p+f_1+r$ is $\mu$-strongly convex.  Then, in the
decomposition
\begin{equation*}
 \mathcal M_y=G_y+r,
 \qquad
 G_y(x)=p(y)+\ip{\nabla p(y)}{x-y}
       +f_1(x)+\delta\norm{x-y}_A^2
\end{equation*}
$G_y$ is convex and
$(L_1+2\delta)$-smooth, while $\mathcal M_y$ is strongly convex with 
parameter
$ \delta+\max\{\delta,\mu\}.$
If $a\in\partial r(z)$, the corresponding certified residual is
\begin{equation*}
 e=\nabla p(y)+\nabla f_1(z)+a+2\delta A(z-y)
 \in\partial\mathcal M_y(z).
\end{equation*}
If $L_1+2\delta=0$, any positive smoothness upper bound may be used in
Proposition~\ref{prop:local-work}.
\end{proposition}

\begin{proof}[Proof of Proposition~\ref{prop:local-curvature}]
The displayed splitting and smoothness bound follow by differentiation.
Convexity of $q$ makes $\mathcal M_y$ $2\delta$-strongly convex.
The descent lemma for $p$ also gives
\[
 R_{2\delta}(y)-R_{2\delta}(x)
 -\ip{\nabla R_{2\delta}(x)}{y-x}
 =\delta\norm{y-x}^2
 -\bigl(p(y)-p(x)-\ip{\nabla p(x)}{y-x}\bigr)
 \ge\frac\delta2\norm{y-x}^2,
\]
so $R_{2\delta}$ is $\delta$-strongly convex.  Together with
$\mu$-strong convexity of $\Phi$, identity
\eqref{eq:model-identity-corrected} shows
that $\mathcal M_y$ is $(\mu+\delta)$-strongly convex as well.  The larger of
the two moduli is $\delta+\max\{\delta,\mu\}$.  The residual formula follows
from the smooth convex sum rule and direct substitution.
\end{proof}

\section{Proofs of the model-error bounds}
\label{app:geometry-proofs}

\subsection{Controlling the model error under Hessian similarity}
\label{app:proof-hilbert-cocoercivity}

\begin{proof}[Proof of Proposition~\ref{prop:hilbert-cocoercivity}]
For any $\beta\ge\delta$, the similarity bound gives
\[
 (\beta-\delta)A
 \preceq\nabla^2R_\beta(z)=\beta A-\nabla^2p(z)
 \preceq(\beta+\delta)A.
\]
Thus $R_\beta$ is convex, and the argument below also shows that
$\delta_{\omega}=\beta+\delta$ is admissible.  With the stated choice
$\beta=2\delta$, the bounds become
\[
 \delta A\preceq\nabla^2R_\beta(z)\preceq3\delta A.
\]
Hence $R_\beta$ is convex and $3\delta$-smooth in the $A$-metric; when
$\delta>0$, it is $\delta$-strongly convex.

Fix $x,y\in\X$, put $u=A^{1/2}(y-x)$, and define
\[
 B:=\int_0^1
 A^{-1/2}\nabla^2R_\beta(x+t(y-x))A^{-1/2}\,dt.
\]
Because $\X$ is convex, the integration path belongs to $\X$, and
$ \delta I\preceq B\preceq3\delta I.$
The fundamental theorem of calculus gives
\[
 A^{-1/2}\bigl(\nabla R_\beta(y)-\nabla R_\beta(x)\bigr)=Bu,
 \qquad
 S_\beta(y,x)=u^\top Bu.
\]
Since $B$ is positive semidefinite and $B\preceq3\delta I$,
$B^2\preceq3\delta B$.  Consequently,
\[
 \Adnorm{\nabla R_\beta(y)-\nabla R_\beta(x)}^2
 =u^\top B^2u
 \le3\delta\,u^\top Bu
 =3\delta S_\beta(y,x).
\]
Thus Assumption~\ref{ass:correction} holds with
$\delta_{\omega}=3\delta$.  Lemma~\ref{lem:model-error-four-point} gives the
corresponding four-point bound.
\end{proof}

\subsection{Controlling the model error in entropy geometry}
\label{app:proof-entropy-cocoercivity}

\begin{proof}[Proof of Proposition~\ref{prop:entropy-cocoercivity}]
Fix $x,y\in\Delta_d^\alpha$, set $u=y-x$, and define the averaged Hessian
\[
 \overline A:=
 \int_0^1\nabla^2R(x+tu)\,dt.
\]
The matrix $\overline A$ is positive semidefinite on directions $w$
satisfying $\mathbf1^\top w=0$.  The fundamental theorem of calculus,
Cauchy--Schwarz, and
the Hessian bound in Proposition~\ref{prop:entropy-cocoercivity} give,
for every tangent vector $v$ with
$\norm{v}_1\le1$,
\begin{align*}
 |\ip{\nabla R(y)-\nabla R(x)}{v}|^2
 =|u^\top\overline A v|^2\le (u^\top\overline A u)(v^\top\overline A v)\le \delta_{\omega}
 \ip{\nabla R(y)-\nabla R(x)}{y-x}.
\end{align*}
Taking the supremum over such $v$ proves the claimed cocoercivity bound.
The lower Hessian bound gives convexity of
$R$ on $\Delta_d^\alpha$.  When $R=R_\beta$, this proves
Assumption~\ref{ass:correction}.
\end{proof}

\subsection{Entropy identities and implementation formulas}
\label{app:entropy-implementation}
\label{app:proof-entropy-identities}

\begin{lemma}[Entropy identities and quotient dual norm]
\label{lem:entropy-identities}
For $e\in\R^d$, the quotient dual norm associated with feasible differences
on the simplex is
\begin{equation}
 \dnorm{e}
 :=\sup_{\substack{\mathbf1^\top v=0\\ \norm{v}_1\le1}}
 \langle e,v\rangle
 =\inf_{c\in\R}\norm{e-c\mathbf1}_\infty
 =\frac12\left(\max_i e_i-\min_i e_i\right).
 \label{eq:quotient-dual}
\end{equation}
The KL identity and Pinsker bound in~\eqref{eq:pinsker} hold on
$\ri\Delta_d$.
\end{lemma}

\begin{proof}[Proof of Lemma~\ref{lem:entropy-identities}]
Expanding the Bregman difference and using
$\sum_i(y_i-x_i)=0$ proves the KL identity.  For completeness, Pinsker's
inequality follows from a binary reduction.  Let
$I=\{i:y_i\ge x_i\}$, write $x(I):=\sum_{i\in I}x_i$, and set
$t=\sum_{i\in I}(y_i-x_i)=\norm{y-x}_1/2$.  The log-sum inequality gives
\[
 \KL(y\Vert x)\ge
 \KL(x(I)+t\,\Vert\,x(I)).
\]
Here the right-hand side is the KL divergence between Bernoulli
distributions with parameters $x(I)+t$ and $x(I)$.
For fixed $b\in(0,1)$, the function
$\KL(a\Vert b)-2(a-b)^2$ is convex in $a$, because its second derivative is
$1/[a(1-a)]-4\ge0$.  Its value and derivative are zero at $a=b$.
Hence the binary divergence is at least $2t^2$, proving
\eqref{eq:pinsker}; boundary cases follow by continuity.

For~\eqref{eq:quotient-dual}, adding $c\mathbf1$ does not change the pairing
with any $v$ satisfying $\mathbf1^\top v=0$, and H\"older's inequality gives
\[
 \sup_{\substack{\mathbf1^\top v=0\\ \norm{v}_1\le1}}\langle e,v\rangle
 \le\inf_{c\in\R}\norm{e-c\mathbf1}_\infty.
\]
Choosing $c=(\max_i e_i+\min_i e_i)/2$ bounds the infimum by half the
range of $e$.  If $i_+$ and $i_-$ attain the maximum and minimum,
respectively, then
$v=(\mathbf e_{i_+}-\mathbf e_{i_-})/2$ satisfies
$\mathbf1^\top v=0$ and $\norm{v}_1\le1$, and gives
$\langle e,v\rangle=(\max_i e_i-\min_i e_i)/2$.  This proves all equalities
in~\eqref{eq:quotient-dual}.
\end{proof}

By~\eqref{eq:quotient-dual}, the residual requirement in
Algorithm~\ref{alg:tiger-corrected} becomes
\[
 \frac12\left(\max_i(e_k)_i-\min_i(e_k)_i\right)\le\sqrt{\xi_k}.
\]
It is invariant under adding any multiple $c\mathbf1$ to $e_k$.

For $x,z\in\Delta_d^\alpha$ and $g\in\R^d$, define
\begin{equation*}
 \widetilde a
 :=\operatorname{softmax}\!\left(
   \frac{\log x+\tau\log z-(\tau/\mu)g}{1+\tau}
  \right),
\end{equation*}
where logarithms are componentwise and
$[\operatorname{softmax}(u)]_i:=e^{u_i}/\sum_j e^{u_j}$.  The constrained
two-center step is
\begin{equation}
 \begin{split}
 a^+
 &:=\argmin_{a\in\Delta_d^\alpha}\left\{
  \frac{\tau}{\mu}\langle g,a\rangle
  +\tau\KL(a\Vert z)+\KL(a\Vert x)
 \right\}=\argmin_{a\in\Delta_d^\alpha}\KL(a\Vert\widetilde a).
 \end{split}
 \label{eq:entropy-two-center-step}
\end{equation}
Indeed, the two objectives in~\eqref{eq:entropy-two-center-step} differ by a
positive factor and a constant independent of $a$.  Thus $a^+$ is the
forward-KL projection of $\widetilde a$ onto $\Delta_d^\alpha$.
Equivalently,
\[
 (a^+)_i=\max\{\alpha,\rho\widetilde a_i\},
 \qquad
 \sum_i\max\{\alpha,\rho\widetilde a_i\}=1,
\]
where the second relation uniquely determines $\rho>0$.

\subsection{Entropic transport facts used by the application}
\label{app:proof-entropic-ot-relative-sc}

\begin{proposition}[Relative curvature of entropic transport]
\label{prop:entropic-ot-relative-sc}
Let $\gamma>0$ and $a\in\Delta_d$ have strictly positive coordinates.
Then the function $W_\gamma(x,a)-\gamma\omega(x)$ is convex in $x$ on
$\Delta_d$.
Consequently, for every $x,y\in\Delta_d^\alpha$ and every
$g_x\in\partial W_\gamma(\cdot,a)(x)$,
\begin{equation}
 W_\gamma(y,a)
 \ge
 W_\gamma(x,a)
 +\langle g_x,y-x\rangle
 +\gamma D_\omega(y,x).
 \label{eq:entropic-ot-relative-sc}
\end{equation}
Here
$ D_\omega(y,x)=\KL(y\Vert x)
 =\sum_{j=1}^d y_j\log\frac{y_j}{x_j}.$
If all entries of $C$ are finite, then $W_\gamma(\cdot,a)$ is infinitely
continuously differentiable on a relative open neighborhood of
$\Delta_d^\alpha$ in the simplex.
\end{proposition}

\begin{proof}
For a nonnegative row $u\in\R_+^d$, define
\[
 \varphi(u):=
 \sum_{\ell=1}^d u_\ell\log u_\ell
 -
 \left(\sum_{\ell=1}^d u_\ell\right)
 \log\left(\sum_{\ell=1}^d u_\ell\right),
 \qquad 0\log0:=0.
\]
This is the closed perspective of the negative entropy on the simplex and
is therefore convex.

For every $\pi\in\mathcal U(x,a)$, its first marginal satisfies
$x_j=\sum_{\ell=1}^d\pi_{j\ell}$.  Hence
\begin{equation}
 W_\gamma(x,a)-\gamma\omega(x)
 =
 \min_{\pi\in\mathcal U(x,a)}
 \left\{
  \langle C,\pi\rangle
  +\gamma\sum_{j=1}^d\varphi(\pi_{j:})
 \right\},
 \label{eq:conditional-entropy-perspective}
\end{equation}
where $\pi_{j:}$ denotes the $j$th row of $\pi$.  This representation is
well defined even when $x_j=0$, because feasibility then forces
$\pi_{j:}=0$ and $\varphi(0)=0$.

Let $x^0,x^1\in\Delta_d$, let
$\pi^0\in\mathcal U(x^0,a)$ and
$\pi^1\in\mathcal U(x^1,a)$, and fix $t\in[0,1]$.  Then
\[
 (1-t)\pi^0+t\pi^1
 \in
 \mathcal U\bigl((1-t)x^0+tx^1,a\bigr).
\]
The objective in
\eqref{eq:conditional-entropy-perspective} is convex in $\pi$.
Taking infima over $\pi^0$ and $\pi^1$ proves that
$W_\gamma(x,a)-\gamma\omega(x)$ is convex in $x$ on $\Delta_d$.

Set $G_a:=W_\gamma(\cdot,a)-\gamma\omega$.  At $x\in\Delta_d^\alpha$, the
smooth-sum rule gives
$g_x-\gamma\nabla\omega(x)\in\partial G_a(x)$.
The subgradient inequality for $G_a$ yields
\[
 W_\gamma(y,a)-\gamma\omega(y)
 \ge W_\gamma(x,a)-\gamma\omega(x)
 +\left\langle g_x-\gamma\nabla\omega(x),y-x\right\rangle.
\]
Rearranging proves~\eqref{eq:entropic-ot-relative-sc}.

For positive marginals and finite $C$, the Sinkhorn representation and the
normalized dual potentials depend smoothly on the first marginal on a
relative neighborhood of $\Delta_d^\alpha$
\citep{cuturi2013sinkhorn,cuturi2014fast,luise2018differential}.
Thus $W_\gamma(\cdot,a)$ is infinitely continuously differentiable on that
relative neighborhood.
\end{proof}

\subsection{Proof for the entropic barycenter application}
\label{app:proof-barycenter}

\begin{proof}[Proof of Lemma~\ref{lem:barycenter-relative-geometry}]
Proposition~\ref{prop:entropic-ot-relative-sc} shows that every function
$f_i(x)-\gamma\omega(x)$ is convex on $\Delta_d$.  Hence
\begin{equation*}
 F(x)-\gamma\omega(x)
 =\frac1n\sum_{i=1}^n(f_i(x)-\gamma\omega(x))
\end{equation*}
is convex on $\Delta_d^\alpha$.

\medskip
\noindent\textit{Relative strong convexity.}
For $x,y\in\Delta_d^\alpha$, convexity of $F-\gamma\omega$ gives
\begin{equation}
 F(y)
 \ge
 F(x)
 +\left\langle\nabla F(x),y-x\right\rangle
 +\gamma D_\omega(y,x).
 \label{eq:barycenter-relative-sc-intrinsic}
\end{equation}
All gradients below are taken along feasible directions of the simplex.

Since $f_1-\gamma\omega$ and $\omega$ are convex, $f_1$ is convex on
$\Delta_d^\alpha$.  Here $\mathcal X=\Delta_d^\alpha$, $r\equiv0$, and
$q$ agrees with $f_1$ on $\mathcal X$.  As stipulated in
Section~\ref{sec:setting}, $q$ is understood to be extended by $+\infty$
outside $\mathcal X$.  The standard subdifferential rule for a
differentiable convex function restricted to a closed convex set therefore
gives, for every $s\in\partial q(x)$,
\begin{equation*}
 \left\langle s-\nabla f_1(x),y-x\right\rangle\le0
 \qquad (y\in\Delta_d^\alpha).
\end{equation*}
Since $p=F-f_1$, it follows that
\begin{align*}
 \left\langle\nabla p(x)+s,y-x\right\rangle
 &=
 \left\langle\nabla F(x),y-x\right\rangle
 +\left\langle s-\nabla f_1(x),y-x\right\rangle\le
 \left\langle\nabla F(x),y-x\right\rangle.
\end{align*}
Combining this inequality with
\eqref{eq:barycenter-relative-sc-intrinsic} proves
Assumption~\ref{ass:outer} with $\mu=\gamma$.

\medskip
\noindent\textit{Well-posedness.}
The set $\Delta_d^\alpha$ is compact, and $\omega$ and the individual
transport losses are smooth on a relative neighborhood of this set.  The
Bregman proximal-gradient subproblem is strongly convex relative to $\omega$
with modulus at least $\gamma+\beta$, while the outer mirror objective is
strongly convex relative to $\omega$ with modulus $1+\tau$.  Consequently,
both subproblems attain unique solutions in $\Delta_d^\alpha$.

\medskip
\noindent\textit{Model-error bound.}
The correction $p=F-f_1$ is $C^2$ on a relative open neighborhood of
$\Delta_d^\alpha$ in the simplex.
Since $\Delta_d^\alpha$ and the set of tangent directions satisfying
$\norm{u}_1=1$ are compact, continuity of $\nabla^2p$ implies that
the maximum defining $M_\alpha$ in~\eqref{eq:barycenter-Malpha} is finite.
By homogeneity,
\[
 \left|u^\top\nabla^2p(z)u\right|
 \le
 M_\alpha\norm{u}_1^2
\]
for every $z\in\Delta_d^\alpha$ and every $u\in\R^d$ satisfying
$\mathbf1^\top u=0$.

For $z\in\Delta_d^\alpha$ and $u\in\R^d$ satisfying
$\mathbf1^\top u=0$, the weighted Cauchy--Schwarz inequality and
$z_j\ge\alpha$ give
\begin{equation*}
 \norm{u}_1^2
 \le u^\top\nabla^2\omega(z)u
 =\sum_{j=1}^d\frac{u_j^2}{z_j}
 \le\frac1\alpha\norm{u}_1^2.
\end{equation*}
Choose
\begin{equation*}
 \beta:=M_\alpha,
 \qquad
 \delta_{\omega}:=\left(1+\frac1\alpha\right)M_\alpha.
\end{equation*}
For every $z\in\Delta_d^\alpha$ and $u\in\R^d$ satisfying
$\mathbf1^\top u=0$,
\begin{align*}
 u^\top\nabla^2R_\beta(z)u
 &\ge M_\alpha\norm{u}_1^2-|u^\top\nabla^2p(z)u|\ge0,\\
 u^\top\nabla^2R_\beta(z)u
 &\le\frac{M_\alpha}{\alpha}\norm{u}_1^2
      +|u^\top\nabla^2p(z)u|
 \le\delta_{\omega}\norm{u}_1^2.
\end{align*}
Thus $R_\beta$ is convex on $\Delta_d^\alpha$ and satisfies the Hessian bound
in Proposition~\ref{prop:entropy-cocoercivity}.
Proposition~\ref{prop:entropy-cocoercivity}
therefore proves Assumption~\ref{ass:correction} on $\Delta_d^\alpha$.
\end{proof}

\begin{corollary}[Explicit barycenter parameters and initialization]
\label{cor:barycenter-explicit-parameters}
Under the conditions of
Lemma~\ref{lem:barycenter-relative-geometry}, let
$x_\alpha^\star$ be the minimizer of $F$ over $\Delta_d^\alpha$,
and let $M_\alpha$ be defined by~\eqref{eq:barycenter-Malpha}.
Then the parameter choices
\begin{equation}
 \mu=\gamma,
 \qquad
 \beta=M_\alpha,
 \qquad
 \delta_\omega=
 \left(1+\frac{1}{\alpha}\right)M_\alpha,
 \qquad
 \tau=
 \left(
  \max\left\{
   1,
   \sqrt{\frac{(1+\alpha^{-1})M_\alpha}{\gamma}}
  \right\}
 \right)^{-1}
 \label{eq:barycenter-geometry-parameters}
\end{equation}
are admissible.

For any initialization $x_0=v_0\in\Delta_d^\alpha$, let
\begin{equation*}
 g_0:=\nabla F(x_0)
 =\frac{1}{n}\sum_{i=1}^n\nabla_xW_\gamma(x_0,a_i)
\end{equation*}
and use the quotient dual norm
\begin{equation*}
 \dnorm{g_0}
 :=
 \inf_{c\in\R}\norm{g_0-c\mathbf1}_\infty
 =
 \frac{1}{2}
 \left(
  \max_j(g_0)_j-\min_j(g_0)_j
 \right).
\end{equation*}
Then
\begin{align}
 \gamma(1+\tau)\KL(x_\alpha^\star\Vert x_0)
 +F(x_0)-F(x_\alpha^\star)
 &\le
 \frac{2(1+\tau)}{\gamma}\dnorm{g_0}^2
 =
 \frac{1+\tau}{2\gamma}
 \left(
  \max_j(g_0)_j-\min_j(g_0)_j
 \right)^2.
 \label{eq:barycenter-initial-bound}
\end{align}
Consequently, the last expression in
\eqref{eq:barycenter-initial-bound} is an admissible value of
$\overline B_0$. The simpler choice
\begin{equation}
 \overline B_0
 =
 \frac{1}{\gamma}
 \left(
  \max_j(g_0)_j-\min_j(g_0)_j
 \right)^2
 \label{eq:barycenter-initial-simple}
\end{equation}
is also admissible and does not depend explicitly on $\tau$.
\end{corollary}

% \begin{proof}
% The proof of Lemma~\ref{lem:barycenter-relative-geometry} establishes the
% finiteness of $M_\alpha$ and the admissibility of the parameter choices in
% \eqref{eq:barycenter-geometry-parameters}.  Since
% $0\in N_{\Delta_d^\alpha}(x_0)$, the vector
% $g_0=\nabla F(x_0)$ belongs to
% $\nabla p(x_0)+\partial q(x_0)$.  Proposition~\ref{prop:initial-bound},
% with $\mu=\gamma$, gives the inequality in
% \eqref{eq:barycenter-initial-bound}; the quotient-dual identity follows
% from~\eqref{eq:quotient-dual}.  Finally,
% $1+\tau_\alpha\le2$ shows that the choice in
% \eqref{eq:barycenter-initial-simple} is also admissible.
% \end{proof}

\begin{proof}
The proof of Lemma~\ref{lem:barycenter-relative-geometry} establishes the
finiteness of $M_\alpha$ and the admissibility of the parameter choices in
\eqref{eq:barycenter-geometry-parameters}. 
The subdifferential rule used above gives
$\nabla f_1(x_0)\in\partial q(x_0)$.  Since
$\nabla F(x_0)=\nabla p(x_0)+\nabla f_1(x_0)$, it follows that
\begin{equation*}
 g_0=\nabla F(x_0)
 \in\nabla p(x_0)+\partial q(x_0).
\end{equation*}
Proposition~\ref{prop:initial-bound},
with $\mu=\gamma$, gives the inequality in
\eqref{eq:barycenter-initial-bound}; the quotient-dual identity follows
from~\eqref{eq:quotient-dual}. Finally,
$1+\tau\le2$ shows that the choice in
\eqref{eq:barycenter-initial-simple} is also admissible.
\end{proof}

\subsection{Additional details for the numerical experiments}
\label{app:fashion-protocol}

\paragraph{Reference objective value.}
We computed a high-accuracy centralized solution
$\widehat x\in\Delta_d^\alpha$ using entropy mirror descent with
backtracking.  The solve was terminated when the truncated-simplex KKT
residual
\[
 \widehat r_{\mathrm{KKT}}
 :=
 \min_{\substack{
  \lambda\in\R,\ \nu\in\R_+^d\\
  \nu_j(\widehat x_j-\alpha)=0,\ j=1,\ldots,d
 }}
 \left\|
  \nabla F(\widehat x)+\lambda\mathbf 1-\nu
 \right\|_\infty
\]
was at most $10^{-7}$.  Here $\lambda$ is the multiplier associated with
$\mathbf 1^\top x=1$, and $\nu$ contains the multipliers associated with
the constraints $x_j\ge\alpha$.  We used the residual-corrected  reference
\[
 F_{\mathrm{ref}}
 :=
 F(\widehat x)-\frac{2\widehat r_{\mathrm{KKT}}^2}{\gamma}.
\]
For exact objective and gradient evaluations, relative strong convexity
implies $F_{\mathrm{ref}}\le F^\star$.  Since the transport costs and
gradients are evaluated by a finite-accuracy Sinkhorn procedure, we use
$F_{\mathrm{ref}}$ as a common numerical reference rather than as a
rigorously certified lower bound.  In the reported problems, the
normalization correction
$2\widehat r_{\mathrm{KKT}}^2/\gamma$ was at most $6.1\cdot10^{-10}$.

\paragraph{Communication accounting.}
One synchronized query corresponds to requesting the agent gradients
needed to form an aggregate gradient at a single point.  A gradient reused
at the same point is counted only once, whereas every new synchronized
evaluation made during backtracking or a rejected trial is counted.
Computations performed locally between synchronized queries, including the
solution of local subproblems, are not charged to the communication count.
The maximum budget was $4000$ synchronized queries for every method; the
figure displays the first $100$ queries, where the differences between the
methods are most visible.

\paragraph{Aggregation of results.}
For each run, we record the relative objective gap after every synchronized
query and plot its running minimum.  The same transport oracle, objective
parameters, and initial point are used for all methods.  Results are
aggregated separately for the within-class and cross-class problems: the
curves show the median over the four problems in each group, and the shaded
regions show the corresponding interquartile range.

\paragraph{Choice of algorithmic parameters.}
A single configuration of each method was selected using a separate
collection of Fashion-MNIST barycenter problems that did not overlap with
the problems reported in Figure~\ref{fig:fashion-all-methods}.  The
configurations were then fixed and used for all reported runs without
task-specific retuning.

\FloatBarrier

\begingroup
\small
\raggedright
\bibliographystyle{unsrtnat}
\bibliography{refs}
\endgroup
\end{document}